\documentclass[11pt]{article}

\usepackage[centertags]{amsmath}
\usepackage{amsfonts}
\usepackage{amssymb}
\usepackage{amsthm}
\usepackage{newlfont}

\newtheorem{thm}{Theorem}[section]
\newtheorem*{thm*}{Theorem}
\newtheorem{cor}[thm]{Corollary}

\newtheorem{lem}[thm]{Lemma}

\newtheorem{prop}[thm]{Proposition}
\newtheorem*{prop*}{Proposition}

\newtheorem*{conj*}{Conjecture}

\newtheorem*{dfn*}{Definition}
\theoremstyle{definition}
\newtheorem{rem}[thm]{\textbf{Remark}}
\newtheorem*{rmk*}{Remark}
\newtheorem*{fact*}{Fact}

\theoremstyle{proof}

\newcommand{\norm}[1]{\left\Vert#1\right\Vert}
\newcommand{\snorm}[1]{\Vert#1\Vert}
\newcommand{\abs}[1]{\left\vert#1\right\vert}
\newcommand{\set}[1]{\left\{#1\right\}}
\newcommand{\brac}[1]{\left(#1\right)}
\newcommand{\scalar}[1]{\left \langle #1 \right \rangle}
\newcommand{\sscalar}[1]{\langle #1 \rangle}
\newcommand{\Real}{\mathbb{R}}

\newcommand{\eps}{\varepsilon}
\newcommand{\E}{\mathcal{E}}

\newlength{\defbaselineskip}
\newcommand{\setlinespacing}[1]%
           {\setlength{\baselineskip}{#1 \defbaselineskip}}

\numberwithin{equation}{section}

\usepackage{bibspacing}

\newcommand{\B}{\mathcal{B}}
\renewcommand{\P}{\mathcal{P}}
\def \F {\mathcal{F}}
\def \J {\mathcal{N}}

\begin{document}

\title{On the role of Convexity in Functional and Isoperimetric Inequalities}

\author{Emanuel Milman\textsuperscript{1}}

\footnotetext[1]{School of Mathematics,
Institute for Advanced Study, Einstein Drive, Simonyi Hall, Princeton, NJ 08540, USA.
Email: emilman@math.ias.edu.\\
Supported by NSF under agreement \#DMS-0635607.}

\maketitle

\begin{abstract}
This is a continuation of our previous work
\cite{EMilmanRoleOfConvexityArxiv}. It is well known that various
isoperimetric inequalities imply their functional ``counterparts'',
but in general this is not an equivalence. We show that under
certain convexity assumptions (e.g. for log-concave probability
measures in Euclidean space), the latter implication can in fact be
reversed for very general inequalities, generalizing a reverse form of Cheeger's inequality
due to Buser and Ledoux. We develop a coherent single framework for passing between
isoperimetric inequalities, Orlicz-Sobolev functional inequalities
and capacity inequalities, the latter being notions introduced by
Maz'ya and extended by Barthe--Cattiaux--Roberto. As an application,
we extend the known results due to the latter authors about the
stability of the isoperimetric profile under tensorization, when
there is no Central-Limit obstruction. As another application, we
show that under our convexity assumptions, $q$-log-Sobolev
inequalities ($q \in [1,2]$) are equivalent to an appropriate family
of isoperimetric inequalities, extending results of Bakry--Ledoux
and Bobkov--Zegarlinski.
Our results extend to the more general setting of Riemannian manifolds with density which satisfy the $CD(0,\infty)$ curvature-dimension condition of Bakry--\'Emery.
\end{abstract}

\section{Introduction}

Let $(\Omega,d,\mu)$ denote a metric probability space. More
precisely, we assume that $(\Omega,d)$ is a separable metric space
and that $\mu$ is a Borel probability measure on $(\Omega,d)$ which
is not a unit mass at a point. Although it is not essential for the
ensuing discussion, it will be more convenient to specialize to the
case where $\Omega$ is a complete smooth oriented $n$-dimensional
Riemannian manifold $(M,g)$, $d$ is the induced geodesic distance,
and $\mu$ is an absolutely continuous measure with respect to the
Riemannian volume form $vol_M$ on $M$. This work continues the study
of interplay between the metric $d$ and the measure $\mu$ initiated
in \cite{EMilmanRoleOfConvexityArxiv}. There are various different
ways to measure this relationship, which may be typically arranged
according to strength, forming a hierarchy. In this work, we will be
primarily concerned with two such different ways.

\subsection{The Hierarchy}

The first way is by means of an isoperimetric inequality. Recall
that Minkowski's (exterior) boundary measure of a Borel set $A
\subset \Omega$, which we denote here by $\mu^+(A)$, is defined as:
\[
 \mu^+(A) := \liminf_{\eps \to 0} \frac{\mu(A_{\eps,d}) - \mu(A)}{\eps}~,
\]
where $A_{\eps,d} := \set{x \in \Omega ; \exists y \in A \;\; d(x,y)
< \eps}$ denotes the $\eps$-extension of $A$ with respect to the
metric $d$. It is clear that this boundary measure is a natural
generalization of the notion of surface area to the metric
probability space setting. An isoperimetric inequality measures the
relation between $\mu^+(A)$ and $\mu(A)$ by means of the
isoperimetric profile $I = I_{(\Omega,d,\mu)}$, defined as the
pointwise maximal function $I : [0,1] \rightarrow \Real_+$, so that:
\begin{equation} \label{eq:I-definition}
\mu^+(A) \geq I(\mu(A)) ~,
\end{equation}
for all Borel sets $A \subset \Omega$. Although it is
possible to guarantee under very general conditions that the function $I$ is symmetric about the
point $1/2$, we will not assume this, and instead define $\tilde{I} = \tilde{I}_{(\Omega,d,\mu)}$ as the function
$\tilde{I}:[0,1/2] \rightarrow \Real_+$ given by:
\[
\tilde{I}(t) := \min(I(t),I(1-t)) ~.
\]

A well known example of an isoperimetric inequality was defined by
Cheeger \cite{CheegerInq}. We will say that our space satisfies
\emph{Cheeger's isoperimetric inequality}, if there exists a
constant $D>0$ so that $\tilde{I}_{(\Omega,d,\mu)}(t) \geq D t$ for
all $t \in [0,1/2]$ ; we denote the best constant $D$ by $D_{Che} =
D_{Che}(\Omega,d,\mu)$. Another useful example pertains to the
standard Gaussian measure $\gamma$ on $(\Real,\abs{\cdot})$, where
$\abs{\cdot}$ is the Euclidean metric. We will say that our space
satisfies a \emph{Gaussian isoperimetric inequality}, if there
exists a constant $D>0$ so that $I_{(\Omega,d,\mu)}(t) \geq D
I_{(\Real,\abs{\cdot},\gamma)}(t)$ for all $t \in [0,1]$ ; we denote
the best constant $D$ by $D_{Gau}$. It is known that
$\tilde{I}_{(\Real,\abs{\cdot},\gamma)}(t) \simeq t \log^{1/2}(1/t)$
uniformly on $t \in [0,1/2]$, where we use the notation $A \simeq B$
to signify that there exist universal constants $C_1,C_2>0$ so that
$C_1 B \leq A \leq C_2 B$. Unless otherwise stated, all of the
constants throughout this work are universal, independent of any
other parameter, and in particular the dimension $n$ in the case of
an underlying manifold. The Gaussian isoperimetric inequality can
therefore be equivalently stated as asserting that there exists a
constant $D>0$ so that $\tilde{I}_{(\Omega,d,\mu)}(t) \geq D t
\log^{1/2}(1/t)$ for all $t \in [0,1/2]$.

\medskip

A second way to measure the interplay between $d$ and $\mu$ is given
by functional inequalities. Let $\F = \F(\Omega,d)$ denote the space
of functions which are Lipschitz on every ball in $(\Omega,d)$ - we
will call such functions ``Lipschitz-on-balls'' - and let $f \in
\F$. We will consider functional inequalities which compare between
$\norm{f}_{N_1(\mu)}$ and $\norm{\abs{\nabla f}}_{N_2(\mu)}$, where
$N_1, N_2$ are some norms associated with the measure $\mu$, like
the $L_p(\mu)$ norms, or some other more general \emph{Orlicz
quasi-norms} associated to the class $\J$ of increasing continuous
functions mapping $\Real_+$ onto $\Real_+$ (see Subsection \ref{subsect:definitions} for
precise definitions). Here, the effect of the metric $d$ is via the
Riemannian metric $g$ which is used to measure $\abs{\nabla
f}:=g(\nabla f,\nabla f)^{1/2}$, although more general ways exist to
define $\abs{\nabla f}$ in the non-manifold setting. There is
clearly no point to test constant functions, so it will be natural
to require that either the expectation $E_\mu f$ or median $M_\mu f$
of $f$ are 0. Here $E_\mu f = \int f d\mu$ and $M_\mu f$ is a value
so that $\mu(f \geq M_\mu f) \geq 1/2$ and $\mu(f \leq M_\mu f) \geq
1/2$.
\begin{dfn*}
We will say that the space $(\Omega,d,\mu)$ satisfies an $(N,q)$
Orlicz-Sobolev inequality ($N \in \J, q \geq 1$) if:
\begin{equation} \label{eq:OS-inq-def}
\exists D>0 \; \text{ s.t. } \; \forall f \in \F \;\;\;\;  D \norm{f
- M_\mu f}_{N(\mu)} \leq \norm{\abs{\nabla f}}_{L_q(\mu)} ~.
\end{equation}
\end{dfn*}
\noindent A similar (yet different) definition was given by Roberto
and Zegarlinski \cite{RobertoZegarlinski} in the case $q=2$
following the work of Maz'ya \cite[p. 112]{MazyaBook}. Our preference to use the median $M_\mu$ in our definition (in place of the more standard expectation $E_\mu$) is immaterial whenever $N$
is a convex function (see Lemma \ref{lem:E-M}).
\medskip

When $N(t)=t^p$, in which case $N(\mu)$ is just the usual $L_p(\mu)$
norm, we will refer to the inequality (\ref{eq:OS-inq-def}) as a
\emph{$(p,q)$ Poincar\'e inequality}. If in addition $M_\mu$ in (\ref{eq:OS-inq-def}) is replaced by
$E_\mu$, the case $p=q=2$ is then just the classical \emph{Poincar\'e inequality}, and we denote the best constant in this inequality by $D_{Poin}$. Similarly, the case $q=1, p=\frac{n}{n-1}$ corresponds to the \emph{Gagliardo--Nirenberg--Sobolev inequality}, and a limitting case when $n$ tends to infinity is the so-called \emph{log-Sobolev inequality}.
More generally, we say that our space satisfies a \emph{$q$-log-Sobolev
inequality} ($q \in [1,2]$), if there exists a constant $D>0$ so
that:
\begin{equation} \label{eq:Intro-q-log-Sob}
\forall f \in \F \;\;\;\;  D \brac{\int |f|^q \log |f|^q d\mu - \int
|f|^q d\mu \log(\int |f|^q d\mu)}^{1/q} \leq \norm{\abs{\nabla
f}}_{L_q(\mu)} ~.
\end{equation}
The best possible constant $D$ above is denoted by $D_{LS_q} =
D_{LS_q}(\Omega,d,\mu)$. Although these inequalities do not
precisely fit into our announced framework, it follows from the work
of Bobkov and Zegarlinski \cite{BobkovZegarlinski} that they are in
fact equivalent to some corresponding Orlicz-Sobolev inequalities
(see Section \ref{sec:general}). Various other functional
inequalities admit an equivalent (up to universal constants)
formulation using an appropriate Orlicz norm $N(\mu)$ on the left
hand side of (\ref{eq:OS-inq-def}). We refer the reader to the
recent paper of Barthe and Kolesnikov \cite{BartheKolesnikov} and
the references therein for an account of several other types of
functional inequalities.

\medskip

It is well known that various isoperimetric inequalities imply their
functional ``counterparts''. It was shown by Maz'ya
\cite{MazyaCheegersInq1,MazyaCheegersInq2} and independently by
Cheeger \cite{CheegerInq}, to whom this is usually attributed, that
Cheeger's isoperimetric inequality implies Poincar\'e's inequality:
$D_{Poin} \geq D_{Che}/2$ (Cheeger's inequality). It was first
observed by M. Ledoux \cite{LedouxBusersTheorem} that a Gaussian
isoperimetric inequality implies a $2$-log-Sobolev inequality:
$D_{LS_2} \geq c D_{Gau}$, for some universal constant $c>0$. This
has been later refined by Beckner (see
\cite{LedouxLectureNotesOnDiffusion}) using an equivalent functional
form of the Gaussian isoperimetric inequality due to S. Bobkov
\cite{BobkovFunctionalFormOfGaussianIsopInq,BobkovGaussianIsopInqViaCube}
(see also  \cite{BartheMaureyIsoperimetricInqs}): $D_{LS_2} \geq
D_{Gau} / \sqrt{2}$. The constants $2$ and $\sqrt{2}$ above are
known to be optimal.

\subsection{Reversing the Hierarchy}

In general, it is known that these implications \emph{cannot} be
reversed. For instance, using $([-1,1],\abs{\cdot},\mu_\alpha)$
where $d\mu_\alpha = \frac{1+\alpha}{2} |x|^\alpha dx$ on $[-1,1]$,
clearly $\mu_\alpha^+([0,1]) = 0$ so $D_{Che} = D_{Gau} = 0$,
whereas one can show that $D_{Poin},D_{LS_2} > 0$ for $\alpha \in
(0,1)$ using criteria for the Poincar\'e and $2$-log-Sobolev
inequalities on $\Real$ due to Kac and Krein
\cite{KacKreinVibratingString} (and also Artola, Talenti and
Tomaselli, cf. Muckenhoupt \cite{MuckenhouptHardyInq}) and Bobkov
and G\"{o}tze \cite{BobkovGotzeLogSobolev}, respectively. We
conclude that in order to have any chance of reversing the above
implications, we will need to add some additional assumptions, which
will prevent the existence of examples as above. As we will see
below, some type of convexity assumptions are a natural candidate.
We start with two important examples when $(M,g) =
(\Real^n,\abs{\cdot})$ and $\abs{\cdot}$ is some fixed Euclidean
norm:
\begin{itemize}
\item
$\Omega$ is an \emph{arbitrary} bounded convex domain in $\Real^n$
($n \geq 2$), and $\mu$ is the uniform probability measure on
$\Omega$.
\item
$\Omega = \Real^n$ ($n \geq 1$) and $\mu$ is an \emph{arbitrary}
absolutely continuous log-concave probability measure, meaning that
$d\mu = \exp(-\psi) dx$ where $\psi: \Real^n \rightarrow \Real \cup
\set{+\infty}$ is convex (we refer to the paper
\cite{Borell-logconcave} of C. Borell for more information).
\end{itemize}

In both cases, we will say that ``our convexity assumptions are
fulfilled''. More generally, we recall the following definition from
\cite{EMilmanRoleOfConvexityArxiv}:

\begin{dfn*}
We will say that our \emph{smooth convexity assumptions} are
fulfilled if:
\begin{itemize}
\item
$(M,g)$ denotes an $n$-dimensional ($n\geq 2$) smooth complete
Riemannian manifold or $(M,g)=(\Real,\abs{\cdot})$, and $\Omega =
M$.
\item
$d$ denotes the induced geodesic distance on $(M,g)$.
\item
$d\mu = \exp(-\psi)  dvol_M$,  $\psi \in C^2(M)$, and as tensor
fields on $M$:
\begin{equation} \label{eq:Intro-BE}
Ric_g + Hess_g \psi \geq 0 ~.
\end{equation}
\end{itemize}
We will say that our \emph{convexity assumptions} are fulfilled if
$\mu$ can be approximated in total-variation by measures
$\set{\mu_m}$ so that $(\Omega,d,\mu_m)$ satisfy our smooth
convexity assumptions.
\end{dfn*}

The condition (\ref{eq:Intro-BE}) is the well-known
Curvature-Dimension condition $CD(0,\infty)$, introduced by Bakry
and \'Emery in their celebrated paper \cite{BakryEmery} (in the more
abstract framework of diffusion generators). Here $Ric_g$ denotes
the Ricci curvature tensor and $Hess_g$ denotes the second covariant
derivative.

\medskip

It is known that under our convexity assumptions, the implications
stated in the previous subsection can be reversed: $D_{Che} \geq c_1
D_{Poin}$ and $D_{Gau} \geq c_2 D_{LS_2}$, for some universal
constants $c_1,c_2>0$. That Cheeger's inequality can be reversed was
first shown by Buser \cite{BuserReverseCheeger} when $\mu$ is
uniform on a closed manifold with $Ric_g \geq 0$, and was recently
strengthened and generalized by Ledoux
\cite{LedouxSpectralGapAndGeometry} to the Bakry--\'Emery abstract
setting, assuming our smooth convexity assumptions. That a
$2$-log-Sobolev inequality implies a Gaussian isoperimetric
inequality under these assumptions was first shown by Bakry and
Ledoux \cite[Section 4]{BakryLedoux} (see also Ledoux
\cite{LedouxSpectralGapAndGeometry}).

\subsection{The Results}

In this work, we generalize all of the above mentioned implications following
Ledoux's diffusion semi-group approach
\cite{LedouxSpectralGapAndGeometry} to a more general framework.
Such a program was initiated in our previous work
\cite{EMilmanRoleOfConvexityArxiv}, where it was first shown how to
use the $CD(0,\infty)$ condition via Ledoux's semi-group gradient
estimates to deduce isoperimetric inequalities from $(p,q)$
Poincar\'e inequalities. Contrary to previous approaches, which
could only deduce isoperimetric information from functional
inequalities with a $\norm{\abs{\nabla f}}_{L_q(\mu)}$ term with
$q=2$ (see \cite[p. 3]{BartheKolesnikov} and the references
therein), it was shown in \cite{EMilmanRoleOfConvexityArxiv} how to
handle arbitrary $q\geq 2$. In the case of $(p,q)$ Poincar\'e
inequalities, an easy reduction step in fact enables one to handle
arbitrary $q \geq 1$. In this work, we show how to deduce
isoperimetric inequalities from very general Orlicz-Sobolev
inequalities in the entire range $q \geq 1$.

The easier case of $q \geq 2$ is handled in Section
\ref{sec:semi-group}, by generalizing our argument for $(p,q)$
Poincar\'e inequalities from \cite{EMilmanRoleOfConvexityArxiv}.
Extending our results to the case $q \geq 1$ (which is very important for applications)
requires additional work, to which end we employ the notion of \emph{capacity}. Capacity
inequalities are certain functional formulations of isoperimetric
inequalities, which were introduced around 1960 by Maz'ya
\cite{MazyaSobolevImbedding,MazyaCapacities}, Federer and Fleming
\cite{FedererFleming}, and used by Bobkov and Houdr\'e in
\cite{BobkovHoudreMemoirs,BobkovHoudre}. Maz'ya's notion of
$q$-capacity for $q=2$ has recently been extended to the metric
probability space setting by Barthe, Cattiaux and Roberto in
\cite{BCRHard} (after being introduced in \cite{BartheRoberto}),
where it was used to deduce isoperimetric inequalities, and has
subsequently appeared in other works as well (e.g.
\cite{BCRSoft,RobertoZegarlinski,SodinLpIsoperimetry,EMilmanSodinIsoperimetryForULC}).
We recall the appropriate definitions in Section
\ref{sec:capacities}, and show that $q$-capacity inequalities are
equivalent in full generality to an appropriate weak-type variant of
these Orlicz-Sobolev inequalities (in the same sense that
$L_{p,\infty}$ is the weak-type $L_p$ quasi-norm). We also give a
very general condition for capacity inequalities to be equivalent to
the usual (non-weak) Orlicz-Sobolev inequalities, which we require for the sequel.
This extends a more restrictive condition (and partly implicit) obtained for $q=2$
in \cite{RobertoZegarlinski}, following a similar condition for
general $q$ in \cite{MazyaBook}.

In Section \ref{sec:general} we use capacities to extend our results
to the whole range $q \geq 1$. We also demonstrate that our
estimates are sharp, by showing that the isoperimetric inequalities
we obtain are in fact equivalent (up to universal constants) to the
functional inequalities used to derive them. To give a taste of the
type of results we obtain, we state the following theorem (see
Theorem \ref{thm:Orlicz-summary} for more details and a slightly
stronger version):

\begin{thm} \label{thm:Intro-Orlicz}
Let $1 \leq q \leq \infty$ and let $N$ denote a Young function, so that:
\begin{equation} \label{eq:Intro-extra}
 \frac{N(t)^{1/q}}{t} \text{ is non-decreasing } ~,~ \exists \alpha > (1/q - 1/2) \vee 0 \;\;\;\;
\frac{N(t^\alpha)}{t} \text{ is non-increasing} ~.
\end{equation}
Then under our convexity assumptions, the following statements are
equivalent:
\begin{enumerate}
 \item
\[
 \forall f \in \F \;\; D_1 \norm{f - M_\mu f}_{N(\mu)} \leq \norm{ \abs{\nabla f} }_{L_q(\mu)}
\]
\item
\[
 \tilde{I}(t) \geq D_2 \frac{t^{1-1/q}}{N^{-1}(1/t)} \;\; \forall t \in [0,1/2] ~,
\]
\end{enumerate}
where the best constants $D_1,D_2$ above satisfy:
\[
c_1 C_{\alpha,q} D_1 \leq D_2 \leq c_2 B_{\alpha,q} D_1 ~,
\]
with $c_1,c_2 > 0$ universal constants and $B_{\alpha,q},
C_{\alpha,q}$ depending explicitly on $\alpha, q$. In fact, the
convexity assumptions are not needed for the direction $(2)
\Rightarrow (1)$, and the assumptions (\ref{eq:Intro-extra}) are not
needed if $q \geq 2$ for the direction $(1) \Rightarrow (2)$.
\end{thm}

When $N(t)=t^2, q=2$, the direction $(2) \Rightarrow (1)$ reduces
(up to constants) to Cheeger's inequality, and the direction $(1)
\Rightarrow (2)$ to its reversed form due to Buser--Ledoux. In
addition, using $N(t) = t^q \log(1 + t^q)$ and a result of Bobkov
and Zegarlinski \cite{BobkovZegarlinski} (generalizing a previous
result of Bobkov and G\"{o}tze \cite{BobkovGotzeLogSobolev}), a
variant of Theorem \ref{thm:Intro-Orlicz} implies (see Corollary \ref{cor:q-log-Sob-implies-isop}) the following:

\begin{thm} \label{thm:Intro-q-log-Sob}
Under our convexity assumptions, the $q$-log-Sobolev inequality (\ref{eq:Intro-q-log-Sob}) (for $q \in [1,2]$) is equivalent to the isoperimetric inequality:
\[
\tilde{I}(t) \geq D_{I_q} t \log^{1/q}(1/t) \;\;\; \forall t \in [0,1/2] ~,
\]
with the best constants $D_{LS_q},D_{I_q}$ satisfying $C_1 D_{LS_q} \leq D_{I_q} \leq C_2 D_{LS_q}$ for some universal constants $C_1,C_2>0$, uniformly on $q \in [1,2]$.
\end{thm}

\noindent
That the latter implies the former was previously shown by Bobkov and Zegarlinski \cite{BobkovZegarlinski} without any convexity assumptions (we prove a more general result in Section \ref{sec:general}). That the former implies the latter for $q=2$ is precisely the statement that $D_{Gau} \geq c D_{LS_2}$ under our convexity assumptions, recovering the previously mentioned result of Bakry--Ledoux \cite{BakryLedoux} and Ledoux \cite{LedouxSpectralGapAndGeometry}.

\medskip

Theorem \ref{thm:Intro-Orlicz} coupled with the equivalence between
Orlicz-Sobolev inequalities and capacity inequalities, enables us to
directly infer isoperimetric inequalities from their $q$-capacity
counterparts under our convexity assumptions. Previous works of Barthe--Roberto
\cite{BartheRoberto}, Barthe--Cattiaux--Roberto \cite{BCRHard} and
Roberto--Zegarlinski \cite{RobertoZegarlinski} have shown that
$2$-capacity inequalities are often equivalent to certain other
types of functional inequalities, such as the
Lata{\l}a--Oleszkiewicz inequality
\cite{LatalaOleszkiewiczBecknerInq} (or more general Beckner-type
inequalities) and additive $\Phi$-Sobolev inequalities. The
advantage of these inequalities compared to the Orlicz-Sobolev
inequalities lies in the fact that they admit tensorization.
To further demonstrate the usefulness of the framework we develop, we easily deduce in Section \ref{sec:tensorization} as a by-product of our methods the dimension-free tensorization results of \cite{BCRHard,BCRSoft}. In fact, we prove the following natural extension of these results.
By the Central-Limit Theorem, one cannot expect a dimension-free result
for isoperimetric profiles which are better than the one for the
Gaussian measure (and even in this case some badly behaved examples
due to Franck Barthe are known
\cite{BartheIsoperimetryUniformDistance}), so some condition needs
to be imposed (we refer to Section \ref{sec:tensorization} for more
details):

\begin{thm} \label{thm:Intro-tensorization}
Let $J:[0,1] \rightarrow \Real_+$ denote an arbitrary continuous
concave function vanishing at $\set{0,1}$ and symmetric about the
point $1/2$. Assume that $J$ does not violate the Central-Limit
obstruction ``with rate $D$'' (in the sense of Theorem
\ref{thm:tensorization}). Let $(M,g)$ denote a Riemannian manifold
equipped with an absolutely continuous Borel probability measure
$\mu$, and assume that:
\[
 I_{(M,g,\mu)}(t) \geq J(t) \;\;\; \forall t \in [0,1] ~.
\]
Then \textbf{without any additional convexity assumptions}, there
exists a constant $c_D>0$ depending only on $D$, such that for any
$k\geq 1$:
\[
I_{(M^{\times k},g^{\otimes k},\mu^{\otimes k})}(t) \geq c_D J(t)
\;\;\; \forall t \in [0,1] ~.
\]
\end{thm}

Here $g^{\otimes k}$ denotes the natural Riemannian product metric
on the product space $M^{\times k}$, and $\mu^{\otimes k}$ denotes
the product measure. This result was previously shown by
Barthe--Cattiaux--Roberto for $J =
I_{(\Real,\abs{\cdot},\mu_\alpha)}$ where $d\mu_\alpha = Z_\alpha
\exp(-|x|^\alpha) dx$ and $\alpha \in [1,2]$ in \cite{BCRHard}, and
for $J = I_{(\Real,\abs{\cdot},\mu_\Phi)}$ where $d\mu_\Phi =
\exp(-\Phi(|x|)) dx$, $\Phi$ is convex and $\sqrt{\Phi}$ is concave,
in \cite{BCRSoft}. It can be shown that $D \leq C$ in the former
case and that $D \leq C_\Phi$ in the latter, so our result recovers
the previous ones. As noted to us by Franck Barthe, it is also possible to approach the proof of Theorem \ref{thm:Intro-tensorization} by pushing further the methods used in \cite{BCRSoft} (perhaps requiring some additional technical assumptions, and with the constant $c_D$ depending on $J$).
However, this approach relies on another class of functional inequalities (``super-Poincar\'e'' inequalities) and a result due to Wang \cite{WangSuperPoincareAndIsoperimetry}, relating between these inequalities and isoperimetric ones. The approach we prefer to undertake, on the other hand, is self-contained, coherent with our framework and requires no further technical assumptions.
To avoid requiring any convexity assumptions in this theorem, we employ a remarkable result of Barthe
\cite{BartheTensorizationGAFA} (see also \cite{RosIsoperimetricProblemNotes}) and a characterization due to
Bobkov \cite{BobkovExtremalHalfSpaces}, which together reduce the
proof to the case $(\Real,\abs{\cdot},\nu)$ with $\nu$ a log-concave density.

\medskip

As already mentioned, our convexity assumptions throughout this work
are used via the semi-group argument described in Section
\ref{sec:semi-group}. More precisely, in that section we assume that
our \emph{smooth} convexity assumptions are fulfilled. To justify
the passage to the limit and conclude that our results are valid
under arbitrary convexity assumptions, we develop a careful
approximation argument in Section \ref{sec:approx}. We emphasize
that this is not just a technical matter, in general it is simply
not true that $(N,q)$ Orlicz-Sobolev inequalities on the spaces
$(\Omega,d,\mu_m)$ are stable under taking limit of $\mu_m$ in the
total-variation norm (see Section \ref{sec:approx}), so the
convexity assumptions will need to be exploited one last time. To
the best of our knowledge, with the exclusion of the tensorization
results above, all the previously known results which were mentioned
did not address this point, and these results were deduced under the
additional smoothness assumptions.

\bigskip

\noindent \textbf{Acknowledgements.} I would like to thank Professor
Jean Bourgain and the Institute for Advanced Study for providing the
perfect research environment. Most especially, I would like to thank
Sasha Sodin for his invaluable help - acquainting me with
capacities, suggesting to look at Ledoux's semi-group argument,
countless other references, many informative conversations and
comments on this manuscript. I am also thankful to Professors Franck
Barthe and Michel Ledoux for their remarks on earlier versions of
this manuscript.


\section{The Semi-Group Argument} \label{sec:semi-group}

In this section, we prove the direction $(1) \Rightarrow (2)$ of Theorem \ref{thm:Intro-Orlicz} for $q \geq 2$.
Our proof is an adaptation of the semi-group argument used in our earlier work \cite{EMilmanRoleOfConvexityArxiv}, which in turn closely follows Ledoux's proof of \cite[Theorem 5.2]{LedouxSpectralGapAndGeometry}.

\subsection{Definitions and Duality} \label{subsect:definitions}

A function $N: \Real_+ \rightarrow \Real_+$ will
be called a Young function if $N(0)=0$ and $N$ is convex increasing. Given a Young function $N$,
the Orlicz norm $N(\mu)$ associated to $N$ is defined as:
\[
 \norm{f}_{N(\mu)} := \inf \set{ v>0 ; \int_\Omega N(|f|/v) d\mu \leq 1}.
\]
For a general increasing continuous function $N : \Real_+
\rightarrow \Real_+$ with $N(0)=0$ and $\lim_{t \rightarrow \infty}
N(t) = \infty$ (we will denote this class by $\J$), the above definition still makes sense, although
$N(\mu)$ will no longer necessarily be a norm. We will say in this
case that it is a quasi-norm.
The following elementary lemma was shown in \cite{EMilmanRoleOfConvexityArxiv}:
\begin{lem} \label{lem:E-M}
Let $N(\mu)$ denote an Orlicz norm associated to the Young function $N$. Then:
\[
 \frac{1}{2} \norm{f - E_\mu f}_{N(\mu)} \leq \norm{f - M_\mu f}_{N(\mu)} \leq 3 \norm{f - E_\mu
f}_{N(\mu)}.
\]
\end{lem}
\noindent This lemma implies that we can pass back and forth between using
the median $M_\mu$ and the expectation $E_\mu$ when excluding
constant functions in our functional inequalities, at the expense of
losing a universal constant.

\begin{dfn*}
We denote by $N^*:\Real_+ \rightarrow \Real_+ \cup \set{+\infty}$ the Legendre-Fenchel transform of $N \in \J$:
\[
 N^*(s) = \sup_{t>0} \{ st - N(t)\} ~.
\]
\end{dfn*}
\noindent $N^*$ is always convex, but unfortunately it may attain the value of $+\infty$, so it will not be
a Young function according to our definition. To avoid this minor issue, it will be more convenient
to work with the dual norm to $N(\mu)$:

\begin{dfn*}
We denote by $N(\mu)^*$ the dual norm to $N(\mu)$, given by:
\[
 \norm{f}_{N(\mu)^*} := \sup \set{ \int f g d\mu ; \norm{g}_{N(\mu)} \leq 1}.
\]
\end{dfn*}
\noindent
Although this will not be used, we comment that it is a nice
exercise (e.g. \cite{RaoRenBook}) to show that when $N$ is a Young
function then:
\[
 \norm{f}_{N^*(\mu)} \leq \norm{f}_{N(\mu)^*} \leq 2 \norm{f}_{N^*(\mu)} ~.
\]
The second inequality is usually called Young's inequality.

\medskip

We borrow the next lemma from \cite[p. 111]{MazyaBook}.
For reasons which will become more apparent in Section \ref{sec:capacities}, we denote by
$N^\wedge : \Real_+ \rightarrow \Real_+$ the function given by $N^\wedge(t) = 1/N^{-1}(1/t)$.

\begin{lem}[Maz'ya] \label{lem:Mazya-duality}
Let $N$ denote a Young function. Then for any Borel set $A$ with $\mu(A)>0$:
\[
\norm{\chi_A}_{N(\mu)^*} = \mu(A) N^{-1}\brac{\frac{1}{\mu(A)}} = \frac{\mu(A)}{N^\wedge(\mu(A))} ~.
\]
\end{lem}
\begin{proof}
On one hand, denoting $g_0 := N^{-1}(1/\mu(A)) \chi_A$, since $\norm{g_0}_{N(\mu)} = 1$ we have:
\[
 \norm{\chi_A}_{N(\mu)^*} \geq \int \chi_A g_0 d\mu = \mu(A) N^{-1}\brac{\frac{1}{\mu(A)}} ~.
\]
On the other hand, by Jensen's inequality, for any $g$ with $\norm{g}_{N(\mu)} \leq 1$, we have:
\begin{eqnarray*}
 \int \chi_A g d\mu &\leq& \mu(A) N^{-1}\brac{\frac{1}{\mu(A)} \int_A N(|g|) d\mu } \\
&\leq& \mu(A) N^{-1}\brac{\frac{1}{\mu(A)} \int_\Omega N(|g|) d\mu } \leq  \mu(A)
N^{-1}\brac{\frac{1}{\mu(A)}} ~.
\end{eqnarray*}
\end{proof}

\subsection{Semi-Group Gradient Estimates}

Given a smooth complete connected Riemannian manifold $\Omega = (M,g)$ equipped with a
probability measure $\mu$ with density $d\mu = \exp(-\psi) dvol_M$, $\psi \in C^2(M,\Real)$, we
define the associated Laplacian $\Delta_{(\Omega,\mu)}$ by:
\begin{equation} \label{eq:Laplacian-def}
 \Delta_{(\Omega,\mu)} := \Delta_{\Omega} - \nabla \psi \cdot \nabla,
\end{equation}
where $\Delta_{\Omega}$ is the usual Laplace-Beltrami operator on $\Omega$. $\Delta_{(\Omega,\mu)}$
acts on $\B(\Omega)$, the space of bounded smooth real-valued functions on $\Omega$.
Let $(P_t)_{t \geq 0}$ denote the semi-group associated to the diffusion process
with infinitesimal generator $\Delta_{(\Omega,\mu)}$
(cf. \cite{DaviesSemiGroupBook,LedouxLectureNotesOnDiffusion}),
characterized by the following system of second order differential
equations:
\[
 \frac{d}{dt} P_t(f) = \Delta_{(\Omega,\mu)} (P_t(f)) \;\;\;\; P_0(f) = f \;\;\; \forall f \in
\B(\Omega)~.
\]
For each $t \geq 0$, $P_t : \B(\Omega) \rightarrow \B(\Omega)$ is a bounded linear operator and
its action naturally extends to the entire $L_p(\mu)$ spaces ($p\geq 1$). We collect several
elementary properties of these operators:
\begin{itemize}
\item
$P_t 1 = 1$.
 \item
$f \geq 0 \Rightarrow P_t f \geq 0$.
\item
$\int P_t f d\mu = \int f d\mu$.
\item
$\abs{P_t(f)}^p \leq P_t(\abs{f}^p)$ for all $p \geq 1$.
\end{itemize}

The following crucial dimension-free reverse Poincar\'e inequality was shown by Bakry and Ledoux in
\cite[Lemma 4.2]{BakryLedoux}, extending Ledoux's approach \cite{LedouxBusersTheorem} for proving
Buser's Theorem (see also \cite[Lemma 2.4]{BakryLedoux}, \cite[Lemma 5.1]{LedouxSpectralGapAndGeometry}):

\begin{lem}[Bakry--Ledoux] \label{lem:Ledoux}
Assume that the following Bakry-\'{E}mery Curvature-Dimension condition holds on $\Omega$:
\begin{equation} \label{eq:Ledoux-condition}
Ric_g + Hess_g \psi \geq -K g ~, K \geq 0 ~.
\end{equation}
Then for any $t\geq 0$ and $f \in \B(\Omega)$, we have:
\[
 c(t) \abs{\nabla P_t f}^2 \leq P_t(f^2) - (P_t f)^2
\]
pointwise, where:
\[
 c(t) = \frac{1 - \exp(-2Kt)}{K} \;\;\; (= 2t \;\; \text{ if } K=0).
\]
\end{lem}

In fact, the proof of this lemma is very general and extends to the abstract
framework of diffusion generators, as developed by Bakry and
\'{E}mery in their celebrated paper \cite{BakryEmery}. In the
Riemannian setting, it is known \cite{QianGradientEstimateWithBoundary} (see also \cite{HsuGradientEstimateWithBoundary,WangGradientEstimateWithBoundary}) that the gradient estimate
of Lemma \ref{lem:Ledoux} is preserved  when restricting to a locally convex set (as defined in
Section \ref{sec:approx}); we refer to Sturm \cite[Proposition 4.15]{SturmCD1} for
a general statement about closedness of the
Barky-\'{E}mery Curvature-Dimension condition
in an arbitrary metric probability space.
The above lemma therefore holds under more general conditions, namely when $\mu$ is supported on a
locally convex domain (connected open set) $\Omega \subset (M,g)$ with $C^2$ boundary, and
$d\mu|_\Omega = \exp(-\psi) dvol_M|_\Omega$, $\psi \in C^2(\overline{\Omega},\Real)$. In this case,
$\Delta_{\Omega}$ in (\ref{eq:Laplacian-def}) denotes the Neumann Laplacian on $\overline{\Omega}$,
$\B(\Omega)$ denotes the space of bounded smooth real-valued functions on $\overline{\Omega}$ satisfying
Neumann's boundary conditions on $\partial \Omega$, and Lemma \ref{lem:Ledoux} remains valid.

Our convexity assumptions are that $K=0$ in Lemma \ref{lem:Ledoux}, and this is what we will
henceforth assume. It is clear that our results in this section as well as Section
\ref{sec:general} may be extended to the case of $K > 0$, but we do not pursue this
direction in this work.

From Lemma \ref{lem:Ledoux}, it is immediate that for any $2 \leq q \leq \infty$:
\begin{equation} \label{eq:L_q-bound}
 \norm{\abs{\nabla P_t f} }_{L_q(\mu)} \leq \frac{1}{\sqrt{2t}} \norm{f}_{L_q(\mu)},
\end{equation}
and using $q = \infty$, Ledoux easily deduces the following dual statement
\cite[(5.5)]{LedouxSpectralGapAndGeometry}:
\begin{cor}[Ledoux] \label{cor:Ledoux}
\begin{equation}
 \norm{ f - P_t f}_{L_1(\mu)} \leq \sqrt{2 t} \norm{ \abs{\nabla f} }_{L_1(\mu)}.
\end{equation}
\end{cor}

\subsection{Orlicz-Sobolev implies Isoperimetry for $q \geq 2$}

\begin{thm} \label{thm:Orlicz-12}
Let $2 \leq q \leq \infty$ and let $N$ denote a Young function. Then under our convexity
assumptions, the statement:
\begin{equation} \label{eq:Orlicz-inq}
 \forall f \in \F \;\; D \norm{f - M_\mu f}_{N(\mu)} \leq \norm{ \abs{\nabla f} }_{L_q(\mu)}
\end{equation}
implies:
\begin{equation} \label{eq:Orlicz-12-conclusion}
 \tilde{I}(t) \geq C_{N,q} D t^{1-1/q} N^\wedge(t) \;\; \forall t \in [0,1/2] ~,
\end{equation}
with $C_{N,q} \geq c > 0$, a universal constant.
\end{thm}

\begin{rem}
We will see how to relax the assumption that $q\geq 2$ to $q \geq 1$ as well as the requirement
that $N$ is convex in Theorem \ref{thm:Orlicz-12-strong}, in which case we will get a different
lower bound on $C_{N,q}$ which will depend on $N$ and $q$.
\end{rem}

\begin{proof}[Proof of Theorem \ref{thm:Orlicz-12}]

We will prove the theorem under the assumption that our measure $\mu$ is supported on a locally
convex domain $\Omega \subset (M,g)$ with $C^2$ boundary, and is of the form $d\mu|_\Omega
= \exp(-\psi) dvol_M|_\Omega$, $\psi \in C^2(\overline{\Omega},\Real)$, as assumed in this section.
The general case follows by an approximation argument detailed in Section \ref{sec:approx}.

Since $N$ is a Young function, we may replace $M_\mu f$ in (\ref{eq:Orlicz-inq})
by $E_\mu f$ using Lemma \ref{lem:E-M} at the expense of an additional universal constant in
the final conclusion.

Let $A$ denote an arbitrary Borel set in $\Omega$, and let $\chi_{A,\eps}(x) :=
(1 - \frac{1}{\eps} d_g(x,A)) \vee 0$ denote a continuous approximation in $\Omega$ to the
characteristic function $\chi_A$ of $A$. Clearly:
\[
 \frac{\mu(A_\eps) - \mu(A)}{\eps} \geq \int \abs{\nabla{\chi_{A,\eps}}} d\mu.
\]
Applying Corollary \ref{cor:Ledoux} to functions in $\B(\Omega)$ which approximate $\chi_{A,\eps}$
(in say $W^{1,1}(\Omega,\mu)$) and passing to the limit
inferior as $\eps \rightarrow 0$, it follows that:
\[
 \sqrt{2t} \mu^+(A) \geq \int \abs{\chi_A - P_t(\chi_A) } d\mu.
\]
We start by rewriting the right hand side above as:
\begin{multline*}
\int_A (1 - P_t(\chi_A)) d\mu + \int_{\Omega \setminus A}
P_t(\chi_A) d\mu = 2\brac{\mu(A) - \int_A
P_t(\chi_A) d\mu} \\
= 2 \brac{\mu(A)(1-\mu(A)) - \int_\Omega (P_t \chi_A - \mu(A))(\chi_A - \mu(A)) d\mu} ~.
\end{multline*}
To estimate the right-most expression, we use the definition of the dual norm:
\[
 \int f g d\mu \leq \norm{f}_{N(\mu)} \norm{g}_{N(\mu)^*} .
\]
Note that we could have also used Young's inequality, yielding $2 \norm{g}_{N^*(\mu)}$ instead of
$\norm{g}_{N(\mu)^*}$ above, but this would lead to slightly worse numeric estimates. Using our assumption (\ref{eq:Orlicz-inq}) with $M_\mu$ replaced by $E_\mu$, we get:
\begin{eqnarray*}
\int_\Omega (P_t \chi_A - \mu(A))(\chi_A - \mu(A)) d\mu &\leq& \norm{P_t \chi_A - \mu(A)}_{N(\mu)} \norm{\chi_A - \mu(A)}_{N(\mu)^*} \\
 &\leq& D^{-1} \norm{\abs{\nabla P_t \chi_A}}_{L_q(\mu)} \norm{\chi_A
- \mu(A)}_{N(\mu)^*}.
\end{eqnarray*}
Using (\ref{eq:L_q-bound}) (recall that $q \geq 2$) to estimate $\norm{\abs{\nabla P_t \chi_A}}_{L_q(\mu)}$, we conclude that:
\begin{equation} \label{eq:sg-proof-conclude}
 \sqrt{2t} \mu^+(A) \geq 2 \brac{\mu(A)(1-\mu(A)) - \frac{1}{\sqrt{2t} D} \norm{\chi_A
- \mu(A)}_{L_{q}(\mu)} \norm{\chi_A - \mu(A)}_{N(\mu)^*}} ~.
\end{equation}
Using Lemma \ref{lem:Mazya-duality}, we estimate $\norm{\chi_A - \mu(A)}_{N(\mu)^*}$:
\begin{eqnarray*}
& & \!\!\!\!\!\!\!\!\!\!\!\!\!\! \norm{\chi_A - \mu(A)}_{N(\mu)^*} \leq (1-\mu(A)) \norm{\chi_A
}_{N(\mu)^*} + \mu(A)\norm{\chi_{\Omega\setminus A} }_{N(\mu)^*} \\
&=& \mu(A)(1-\mu(A))  \brac{\frac{1}{N^{\wedge}(\mu(A))} + \frac{1}{N^{\wedge}(1-\mu(A))}}  \leq
2 \frac{\mu(A)(1-\mu(A)) } {N^\wedge(\min(\mu(A),1-\mu(A)))} ~.
\end{eqnarray*}
We also have the following rough estimate (for $q \geq 2$):
\[
\norm{\chi_A - \mu(A)}_{L_{q}(\mu)} \leq \brac{\mu(A)(1-\mu(A))}^{1/q} ~.
\]
It remains to optimize on $t$. Evaluating (\ref{eq:sg-proof-conclude}) at time:
\[
 t = \frac{8 \brac{\mu(A)(1-\mu(A))}^{2/q} }{D^2 N^\wedge(\min(\mu(A),1-\mu(A)))} ~,
\]
we deduce:
\begin{eqnarray*}
\mu^+(A) &\geq & \frac{D}{4} \brac{\mu(A)(1-\mu(A))}^{1-1/q} N^\wedge(\min(\mu(A),1-\mu(A))) \\
&\geq& \frac{D}{4 \cdot 2^{1-1/q}} \min(\mu(A),1-\mu(A))^{1-1/q} N^\wedge(\min(\mu(A),1-\mu(A))) ~.
\end{eqnarray*}
This concludes the proof.
\end{proof}

\begin{rem} \label{rem:general-smoothness}
As evident from the proof, the definition of smooth convexity assumptions given in the Introduction may
be extended to encompass the more general case treated in this section. Consequently, the same remark applies
to all of the subsequent results which employ our convexity assumptions.
\end{rem}


\section{Capacities} \label{sec:capacities}

As already mentioned in the Introduction, $q$-capacity inequalities
are certain functional formulations of isoperimetric inequalities.
We conform to the definition given in
\cite{EMilmanRoleOfConvexityArxiv}, which is
a variation on the definition introduced
by Maz'ya \cite{MazyaCapacities,MazyaBook} (for general $q$) and extended by Barthe, Cattiaux and
Roberto (with $q=2$) in \cite{BCRHard} (after being introduced in \cite{BartheRoberto}).
In this section, we introduce a coherent unified framework which
provides an equivalence between capacity inequalities and weak-type
Orlicz-Sobolev functional inequalities (introduced below), and a
general sufficient condition for an equivalence to Orlicz-Sobolev
inequalities. We also provide an argument for handling general
metric probability spaces. There is essentially no novel content in
some parts of this section, and these are provided here for
completeness.

\subsection{Definitions}

\begin{dfn*}
Given a metric probability space $(\Omega,d,\mu)$, $1 \leq q<\infty$ and $0 \leq a \leq b \leq 1$,
we denote:
\[
 Cap_q(a,b) := \inf\set{ \norm{\abs{\nabla \Phi}}_{L_q(\mu)} ; \mu\set{\Phi=1} \geq a \;,\;
\mu\set{\Phi=0} \geq 1-b },
\]
where the infimum is on all $\Phi : \Omega \rightarrow [0,1]$ which are Lipschitz-on-balls.
\end{dfn*}

\begin{rem}
Both Maz'ya's definition \cite{MazyaBook} for general $q$ and the definition of
Barthe--Cattiaux--Roberto \cite{BartheRoberto,BCRHard} for the case $q=2$ use $\int
\abs{\nabla \Phi}^q d\mu$ instead of our normalized $\norm{\abs{\nabla \Phi}}_{L_q(\mu)}$. Our
definition seems more convenient, as witnessed by the formulation of our results below.
\end{rem}

\begin{rem}
The use of the metric $d$ induced by the geodesic distance on $(M,g)$ was essential for applying
the (linear) semi-group argument of the previous section. Throughout this section, as well as the
relevant parts of Sections \ref{sec:general} and \ref{sec:tensorization}, such a restriction no
longer exists, and one may use an arbitrary metric $d$. In this case, we interpret $\abs{\nabla f}$ for any $f \in \F$ as the following Borel function:
\[
 \abs{\nabla f}(x) := \limsup_{d(y,x) \rightarrow 0+} \frac{|f(y) - f(x)|}{d(x,y)} ~.
\]
(and we define it as 0 if $x$ is an isolated point - see \cite[pp. 184,189]{BobkovHoudre}
for more details).
\end{rem}

\begin{rem} \label{rem:cap-app0}
A remark which will be useful for dealing with general metric
probability spaces, is that in the definition of capacity, we may
always assume that $\int_{\set{\Phi=t}} \abs{\nabla \Phi}^q d\mu =
0$, for any $t \in (0,1)$, even though we may have
$\mu\set{\Phi=t}>0$. The argument is as follows.

Denote $\Gamma:=\set{t \in (0,1) ; \mu\set{\Phi=t}>0}$ the discrete
countable set of atoms of $\Phi$ under $\mu$, and write $\Gamma =
\set{\gamma_i}_{i=-A,\ldots,B}$, $A,B \in \set{0,1,\ldots,\infty}$,
with $\gamma_i < \gamma_{i+1}$ (and set $\gamma_{-(A+1)} = 0$ if $A
< \infty$ and $\gamma_{B+1} = 1$ if $B<\infty$). Denote $\beta_i =
(\gamma_i + \gamma_{i+1})/2$, and set for $\eps>0$:
\[
\Phi_\eps(x) := \begin{cases} \Phi(x) & \Phi(x) \in \Gamma \\
\brac{\brac{(1+\eps)(\Phi(x) - \beta_i) + \beta_i} \vee \gamma_i}
\wedge \gamma_{i+1}  & \gamma_i < \Phi(x) < \gamma_{i+1}
\end{cases} ~.
\]
Clearly $\Phi_\eps \in \F$ and $\norm{\abs{\nabla
\Phi_\eps}}_{L_q(\mu)} \leq (1+\eps)\norm{\abs{\nabla
\Phi}}_{L_q(\mu)}$, so $\Phi_\eps$ is a valid approximation. Since
$\Phi$ is Lipschitz-on-balls and $\Phi_\eps$ has the same set of
atoms $\Gamma$ as $\Phi$, it is immediate to verify that for every
$\gamma_i \in \Gamma$:
\begin{equation} \label{eq:cap-app0}
 \int_{\set{\Phi_\eps = \gamma_i}} \abs{\nabla \Phi_\eps}^q d\mu =
 \int_{\set{\Phi = \mu_i} \cup \set{\Phi = \nu_i}} \abs{\nabla \Phi_\eps}^q
d\mu ~,
\end{equation}
where:
\[
\mu_i = \beta_{i-1} + \frac{\gamma_i - \beta_{i-1}}{1+\eps} ~ \quad
, \quad ~ \nu_i = \beta_{i} + \frac{\gamma_i - \beta_i}{1+\eps} ~.
\]
But the integral on the right hand side of (\ref{eq:cap-app0}) is 0
since $\mu_i,\nu_i \notin \Gamma$.
\end{rem}

The following proposition (see \cite{MazyaCapacities}, \cite{FedererFleming}, \cite{BobkovHoudre}, \cite[Proposition A]{SodinLpIsoperimetry}) encapsulates the connection between capacity and the isoperimetric profile $I = I_{(\Omega,d,\mu)}$ (we refer to \cite{EMilmanRoleOfConvexityArxiv} for a careful proof).
\begin{prop}[Maz'ya, Federer--Fleming, Bobkov--Houdr\'e] \label{prop:Sodin-prop}
For all $0<a < b < 1$:
\begin{equation} \label{eq:Sodin-prop}
 \inf_{a \leq t \leq b} I(t) \leq Cap_1(a,b) \leq \inf_{a \leq t < b}
I(t) ~.
\end{equation}
\end{prop}

\noindent Since obviously $Cap_1(a,b) = Cap_1(1-b,1-a)$, we have the following useful corollary: 

\begin{cor} \label{cor:cap1}
For any non-decreasing continuous function $J:[0,1/2]\rightarrow \Real_+$:
\[
 \tilde{I}(t) \geq J(t) \;\; \forall t \in [0,1/2] \;\; \iff \;\; Cap_1(t,1/2)
\geq J(t) \;\; \forall t \in [0,1/2] ~.
\]
\end{cor}

\medskip

\begin{dfn*}
Given $N \in \J$, denote by $N^\wedge : \Real_+ \rightarrow \Real_+$ the ``adjoint'' function:
\[
 N^\wedge(t) := \frac{1}{N^{-1}(1/t)}.
\]
\end{dfn*}

\begin{rem} \label{rem:J-wedge}
Note that the operation $N \rightarrow N^\wedge$ is an involution on $\J$, and that
$N(\cdot^\alpha)^\wedge = (N^\wedge)^{1/\alpha}$ for $\alpha > 0$.
\end{rem}

\begin{lem}
$N(t^\alpha)/t$ is non-decreasing iff $N^\wedge(t)^{1/\alpha}/t$ is non-increasing ($\alpha>0$).
\end{lem}
\begin{proof}
It is enough to prove the ``only if'' direction for $\alpha=1$ by Remark \ref{rem:J-wedge}. Our
assumption is that for all $0 < t_1 \leq t_2$:
\[
 N(t_1) / t_1 \leq N(t_2) / t_2 ~.
\]
Let $s_1 \geq s_2 > 0$ be given. Using $t_i = N^{-1}(1/s_i)$, $i=1,2$ above (which is legitimate
since $N$ is increasing), we deduce:
\[
N^\wedge(s_1) / s_1 \leq N^\wedge(s_2) / s_2 ~,
\]
as required.
\end{proof}

We denote by $L_{s,\infty}(\mu)$ the weak $L_s$ quasi-norm, defined as:
\[
 \norm{f}_{L_{s,\infty}(\mu)} := \sup_{t>0} \mu(|f| \geq t)^{1/s} t.
\]
We now extend the definition of the weak $L_{s}$ quasi-norm to Orlicz quasi-norms $N(\mu)$, using the adjoint
function $N^\wedge$:
\begin{dfn*}
Given $N\in \J$, define the weak $N(\mu)$ quasi-norm as:
\[
 \norm{f}_{N(\mu),\infty} := \sup_{t>0} N^\wedge(\mu\set{\abs{f} \geq t}) t.
\]
\end{dfn*}
\noindent This definition is consistent with the one for $L_{s,\infty}$, and satisfies:
\begin{equation} \label{eq:weak-norm}
 \norm{f}_{N(\mu),\infty} \leq \norm{f}_{N(\mu)} ~,
\end{equation}
as easily checked using the Markov-Chebyshev inequality. Also note that by a simple
union-bound:
\[
\norm{f+g}_{N(\mu),\infty} \leq 2 \brac{\norm{f}_{N(\mu),\infty} + \norm{g}_{N(\mu),\infty}} ~.
\]

\begin{rem}
The motivation for the definition of $N^\wedge$ stems from the immediate observation that for any
Borel set $A$:
\[
 \norm{\chi_A}_{N(\mu)} =  \norm{\chi_A}_{N(\mu),\infty} = N^\wedge(\mu(A)) ~.
\]
For this reason, the expression $1 / N^{-1}(1/t)$ already appears in the works of Maz'ya \cite[p. 112]{MazyaBook}
and Roberto--Zegarlinski \cite{RobertoZegarlinski}.
\end{rem}

\begin{dfn*}
An inequality of the  form:
\begin{equation} \label{eq:Lq-implies-capq}
\forall f \in \F \;\;\; D \norm{f - M_\mu f}_{N(\mu),\infty} \leq
\norm{\abs{\nabla f}}_{L_q(\mu)}
\end{equation}
is called a weak-type Orlicz-Sobolev inequality.
\end{dfn*}

\begin{lem} \label{lem:Lq-implies-capq}
The weak-type Orlicz-Sobolev inequality (\ref{eq:Lq-implies-capq}) implies:
\[
 Cap_q(t,1/2) \geq D N^\wedge(t) \;\;\; \forall t \in [0,1/2]
\]
\end{lem}

\begin{proof}
Apply (\ref{eq:Lq-implies-capq}) to $f = \Phi$, where $\Phi : \Omega \rightarrow
[0,1]$ is any Lipschitz-on-balls function so that $\mu\set{\Phi=1} \geq t$ and
$\mu\set{\Phi=0} \geq 1/2$. Since $M_\mu \Phi = 0$, it
follows that:
\[
 \norm{\abs{\nabla \Phi}}_{L_q(\mu)} \geq D \norm{\Phi}_{N(\mu),\infty} \geq D
N^\wedge(\mu(\set{\Phi=1})) \geq D N^\wedge(t),
\]
Taking the infimum
over all $\Phi$ as above, the assertion is verified.
\end{proof}

\subsection{Equivalences}

\begin{prop} \label{prop:Capq-Lq-weak}
Let $1 \leq q<\infty$, then the following statements are equivalent:
\begin{enumerate}
 \item
\begin{equation} \label{eq:Capq-implies-Lq-weak}
 \forall f \in \F \;\;\; D_1 \norm{f - M_\mu f}_{N(\mu),\infty} \leq
\norm{\abs{\nabla f}}_{L_q(\mu)} \;\; ,
\end{equation}
\item
\[
 Cap_q(t,1/2) \geq D_2 N^\wedge(t) \;\;\; \forall t \in [0,1/2] \;\; ,
\]
\end{enumerate}
and the best constants $D_1,D_2$ above satisfy $ D_1 \leq D_2 \leq 4 D_1$.
\end{prop}
\begin{proof}
$D_2 \geq D_1$ by Lemma \ref{lem:Lq-implies-capq}. To see the other
direction, note that by approximating $f$ (as in Remark
\ref{rem:cap-app0}), we may assume that $\int_{\set{f=t}}
\abs{\nabla f}^q d\mu = 0$ for all $t \in \Real$, and by replacing
$f$ with $f - M_\mu f$, that $M_\mu f = 0$. Note that if suffices to
show (\ref{eq:Capq-implies-Lq-weak}) with $D_1 = D_2$ for
non-negative functions for which $\mu\set{f=0} \geq 1/2$, since for
a general function as above, we can apply
(\ref{eq:Capq-implies-Lq-weak}) to $f_+ = f \chi_{f \geq 0}$ and to
$f_- = -f \chi_{f \leq 0}$, which yields:
\begin{multline*}
\norm{\abs{\nabla f}}_{L_q(\mu)} =
\brac{\int \abs{\nabla f_+}^q d\mu + \int \abs{\nabla f_-}^q d\mu}^{1/q} \geq
D_1 \brac{\norm{f_+}^q_{N(\mu),\infty} + \norm{f_-}^q_{N(\mu),\infty}}^{1/q} \\
\geq D_1 2^{1/q-1} \brac{
\norm{f_+}_{N(\mu),\infty} +
\norm{f_-}_{N(\mu),\infty}} \geq \frac{D_1}{4}
\norm{f}_{N(\mu),\infty} ~.
\end{multline*}

Given a non-negative function $f$ as above ($\mu\set{f=0} \geq 1/2$ hence $M_\mu f = 0$), and $t>0$, define $\Omega_t = \set{f \leq
t}$ and $f_t := f/t \wedge 1$. Then:
\begin{eqnarray*}
 \brac{\int_\Omega \abs{\nabla f}^q d\mu}^{1/q} &\geq& \brac{\int_{\Omega_t} \abs{\nabla f}^q
d\mu}^{1/q} \geq t \brac{\int_{\Omega} \abs{\nabla f_t}^q d\mu}^{1/q} \\
&\geq& t Cap_q(\mu\set{f_t \geq 1},1/2) \geq D_2 t N^\wedge(\mu\set{f \geq
t}) ~.
\end{eqnarray*}
Taking supremum on $t>0$, the assertion follows.
\end{proof}

\begin{prop} \label{prop:Capq-Lq}
If $N(t)^{1/q}/t$ is non-decreasing on $\Real_+$ with $1 \leq q<\infty$, then the following statements are
equivalent:
\begin{enumerate}
 \item
\begin{equation} \label{eq:Capq-implies-Lq}
 \forall f \in \F \;\;\; D_1 \norm{f - M_\mu f}_{N(\mu)} \leq
\norm{\abs{\nabla f}}_{L_q(\mu)} \;\; ,
\end{equation}
\item
\[
 Cap_q(t,1/2) \geq D_2 N^\wedge(t) \;\;\; \forall t \in [0,1/2] \;\; ,
\]
\end{enumerate}
and the best constants $D_1,D_2$ above satisfy $ D_1 \leq D_2 \leq 4 D_1$.
\end{prop}
\begin{rem}
As already mentioned in the Introduction, we call an inequality of the form
(\ref{eq:Capq-implies-Lq}) an Orlicz-Sobolev inequality (even though $N$ may not be convex).
\end{rem}
\begin{rem}
One may show (see e.g. the proof of \cite[Theorem 1]{RobertoZegarlinski}) that when $N(t^{1/q})$ is
convex (so in particular $N(t)^{1/q}/t$ is non-decreasing),
Proposition \ref{prop:Capq-Lq} is equivalent to a theorem of Maz'ya
\cite[p. 112]{MazyaBook}, but there the condition on $N$ is hidden.
Such a stronger assumption is too restrictive for our purposes.
Under this stronger assumption, the statement of this proposition was used in the case $q=2$ in \cite{RobertoZegarlinski} and for $N(t)=t^{2},q=2$ in \cite{BCRHard}.
\end{rem}

\begin{proof}
$D_2 \geq D_1$ by (\ref{eq:weak-norm}) and Lemma
\ref{lem:Lq-implies-capq}. To see the other direction, we assume
again (as in Remark \ref{rem:cap-app0}) that $\int_{\set{f=t}}
\abs{\nabla f}^q d\mu = 0$ for all $t \in \Real$, and by replacing
$f$ with $f - M_\mu f$, that $M_\mu f = 0$. Again, if suffices to
show (\ref{eq:Capq-implies-Lq}) for non-negative functions for which
$\mu\set{f=0} \geq 1/2$, but now we do not lose in the constant.
Indeed, for a general function as above, we can apply
(\ref{eq:Capq-implies-Lq}) to $f_+ = f \chi_{f \geq 0}$ and to $f_-
= -f \chi_{f \leq 0}$, which yields:
\[
\norm{\abs{\nabla f}}^q_{L_q(\mu)} =
\int \abs{\nabla f_+}^q d\mu + \int \abs{\nabla f_-}^q d\mu \geq
D_1^q \brac{\norm{f_+}^q_{N(\mu)} + \norm{f_-}^q_{N(\mu)}} \geq
 D_1^q \norm{f}^q_{N(\mu)}.
\]
The last inequality follows from the fact that 
$N^{1/q}(t)/t$ is non-decreasing, so denoting $v_{\pm} = \norm{f_{\pm}}_{N(\mu)}$,
we indeed verify that:
\begin{multline*}
\int N\brac{\frac{f_+ + f_-}{(v^q_+ + v^q_-)^{1/q}}} d\mu =
 \int N\brac{\frac{f_+}{v_+} \frac{v_+}{(v^q_+ + v^q_-)^{1/q}}} d\mu +
\int N\brac{\frac{f_-}{v_-} \frac{v_-}{(v^q_+ + v^q_-)^{1/q}}} d\mu \\
\leq
\frac{v_+^q}{v^q_+ + v^q_-}  \int  N\brac{\frac{f_+}{v_+}}d\mu +
\frac{v_-^q}{v^q_+ + v^q_-}  \int  N\brac{\frac{f_-}{v_-}}d\mu \leq 1 ~.
\end{multline*}

We will first assume that $f$ is bounded. Given a bounded
non-negative function $f$ as above ($M_\mu f = 0$ and $\mu\set{f=0}
\geq 1/2$), we may assume by homogeneity that $\norm{f}_{L_\infty} =
1$. For $i\geq 1$, denote $\Omega_i = \set{1/2^{i} \leq f \leq
1/2^{i-1}}$, $m_i = \mu(\Omega_i)$, $f_i = 2^{i}(f - 1/2^{i}) \vee 0
\wedge 1$ and set $m_0=0$. Also denote $J := N^\wedge$. Now:
\begin{multline*}
\norm{\abs{\nabla f}}_{L_q(\mu)}^q = \sum_{i=1}^\infty
\int_{\Omega_i} \abs{\nabla f}^q d\mu \geq
\sum_{i=1}^\infty \frac{1}{2^{qi}} \int_{\Omega} \abs{\nabla f_i}^q d\mu \\
\geq \sum_{i=1}^\infty
\frac{1}{2^{q i}} Cap^q_q(\mu\set{f \geq 1/2^{i-1}},1/2) \geq D_2^q \sum_{i=2}^\infty
\frac{J^q(m_{i-1})}{2^{q i}} = \frac{D_2^q}{4^q} V^q,
\end{multline*}
where:
\[
V := \brac{\sum_{i=1}^\infty \frac{J^q(m_{i})}{2^{q (i-1)}}}^{1/q} ~.
\]
It remains to show that $\norm{f}_{N(\mu)} \leq V$. Indeed:
\[
 \int_\Omega N\brac{\frac{f}{V}} d\mu \leq \sum_{i=1}^\infty m_i N\brac{\frac{1}{2^{i-1}
V}}
= \sum_{i=1}^\infty \frac{J^{-1}(J(m_i))}{J^{-1}(2^{i-1} V)} \leq \sum_{i=1}^\infty
\frac{J^q(m_i)}{2^{q(i-1)} V^q} = 1,
\]
where in the last inequality we have used the fact that $N(t)^{1/q}/t$ is non-decreasing, hence
$(J^{-1})^{1/q}(t)/t$ is non-decreasing, and therefore:
\[
\frac{ J^{-1}(x) } {J^{-1}(y) } \leq \brac{\frac{x}{y}}^q,
\]
whenever $x/y \leq 1$, which is indeed the case for us.

For a non-bounded $f\in \F$ with $\mu\set{f=0} \geq 1/2$, we may
define $f_m = f \wedge b_m$ so that $\mu\set{f
> b_m} \leq 1/m$ and (just for safety) $\mu\set{f=b_m} = 0$. It then follows
by what was proved for bounded functions that:
\[
\norm{\abs{\nabla f}}_{L_q(\mu)} \geq \lim_{m \rightarrow \infty}
\norm{\abs{\nabla f_m}}_{L_q(\mu)} \geq D_1 \lim_{m \rightarrow
\infty} \norm{f_m}_{N(\mu)} = D_1 Z ~,
\]
where all limits exist since they are non-decreasing. To conclude,
$Z \geq \norm{f}_{N(\mu)}$, since $N$ is continuous, so by
the Monotone Convergence Theorem:
\[
\int N(f/Z) d\mu = \int \lim_{m \rightarrow \infty}
N(f_m/Z) d\mu = \lim_{m \rightarrow \infty} \int
N(f_m/Z) d\mu \leq 1 ~.
\]
\end{proof}

\medskip

We immediately deduce from Propositions \ref{prop:Capq-Lq-weak} and \ref{prop:Capq-Lq} the following peculiar corollary on the equivalence of the weak and usual Orlicz norms for some functional inequalities:

\begin{cor}
Let $N \in \J$, and assume that $N(t)^{1/q}/t$ is non-decreasing on $\Real_+$
with $1 \leq q <\infty$. Then the following statements are equivalent:
\begin{enumerate}
 \item
\[
 \forall f \in \F \;\;\; D_1 \norm{f - M_\mu f}_{N(\mu)} \leq
\norm{\abs{\nabla f}}_{L_q(\mu)} \;\; ,
\]
\item
\[
 \forall f \in \F \;\;\; D_2 \norm{f - M_\mu f}_{N(\mu),\infty} \leq
\norm{\abs{\nabla f}}_{L_q(\mu)} \;\; ,
\]
\end{enumerate}
and the best constants $D_1,D_2$ above satisfy $ D_1 \leq D_2 \leq 4 D_1$.
\end{cor}

\begin{rem}
This corollary seems useful, even in the case of $F(t) = t^2$ and $q=2$, where this amounts to an
equivalent characterization of the classical Poincar\'e inequality, using the weak $L_{2,\infty}$
quasi-norm on the left hand side. We do not know whether this characterization was previously
noticed.
\end{rem}

Another useful fact which follows from Propositions \ref{prop:Capq-Lq-weak} and \ref{prop:Capq-Lq} is that the
behavior of $N$ at a neighborhood of 0 is simply irrelevant as far as Orlicz-Sobolev inequalities are concerned:

\begin{cor} \label{cor:N-at-0}
Let $N \in \J$, and assume that $N(t)^{1/q}/t$ is non-decreasing on $\Real_+$
with $1\leq q<\infty$. Define:
\[
N_0(t) = \begin{cases} 2 (t/N^{-1}(2))^q  & t \in [0,N^{-1}(2)] \\ N(t) & t \in [N^{-1}(2),\infty) \end{cases} ~.
\]
Then the following statements are equivalent:
\begin{enumerate}
 \item
\[
 \forall f \in \F \;\;\; D_1 \norm{f - M_\mu f}_{N(\mu)} \leq
\norm{\abs{\nabla f}}_{L_q(\mu)} \;\; ,
\]
\item
\[
 \forall f \in \F \;\;\; D_2 \norm{f - M_\mu f}_{N_0(\mu)} \leq
\norm{\abs{\nabla f}}_{L_q(\mu)} \;\; ,
\]
\end{enumerate}
and the best constants $D_1,D_2$ above satisfy $ \frac{1}{4} D_1 \leq D_2 \leq 4 D_1$.
\end{cor}
\begin{proof}
Note that $N_0$ still satisfies that $N_0(t)^{1/q}/t$ is non-decreasing and that
$N^\wedge(t) = N_0^\wedge(t)$ on $t \in [0,1/2]$.
Using Proposition \ref{prop:Capq-Lq-weak} to pass from the Orlicz-Sobolev inequality to a capacity inequality,
we can then exchange between $N$ and $N_0$, and use Proposition \ref{prop:Capq-Lq} to pass back to the other
Orlicz-Sobolev inequality.
\end{proof}


\section{The General Theorem} \label{sec:general}

Note that the assumption $q \geq 2$ was needed for the proof of
Theorem \ref{thm:Orlicz-12} in order to use the estimate
(\ref{eq:L_q-bound}), and the convexity of $N$ was needed to employ Lemma \ref{lem:Mazya-duality}.
In order to relax these assumptions, as well as to deduce the direction
$(2) \Rightarrow (1)$ in Theorem \ref{thm:Intro-Orlicz}, we will need some additional observations,
which are most-naturally formulated in the language of capacities.

\subsection{Passing between $q$-capacities}

In the following proposition, the case $q_0=1$ is due to Maz'ya
\cite[p. 105]{MazyaBook}.
Motivated by the method used in our joint work with Sodin in
\cite{EMilmanSodinIsoperimetryForULC}, we provide an independent
proof, which generalizes to the case of an arbitrary metric
probability space and $q_0 > 1$. We denote the conjugate exponent to
$q \in [1,\infty]$ by $q^* = q/(q-1)$.

\begin{prop} \label{prop:increase-Orlicz-q}
Let $1 \leq q_0 \leq q < \infty$ and set $p_0 = q_0^*, p = q^*$. Then for all $0<a<b<1$:
\[
 \frac{1}{Cap_q(a,b)} \leq \gamma_{p,p_0} \brac{\int_a^b
\frac{ds}{(s-a)^{p/p_0} Cap_{q_0}^p(s,b)}}^{1/p} ~,
\]
where:
\begin{equation} \label{eq:gamma}
\gamma_{p,p_0} :=  \frac{(\frac{p_0}{p} - 1)^{1/p_0}}{(1-\frac{p}{p_0})^{1/p}} ~.
\end{equation}
\end{prop}

\begin{proof}
Let $0<a<b<1$ be given, and let $\Phi : \Omega \rightarrow [0,1]$ be
a function in $\F$ such that $a':=\mu\set{\Phi=1} \geq a$ and
$1-b':=\mu\set{\Phi=0} \geq 1-b$. As usual (see Remark
\ref{rem:cap-app0}), by approximating $\Phi$, we may assume that
$\int_{\set{\Phi=t}} \abs{\nabla \Phi}^q d\mu = 0$ for all $t \in
(0,1)$. Let $C := \set{ t \in (0,1) ; \mu\set{\Phi = t} > 0 }$
denote the discrete set of atoms of $\Phi$ under $\mu$, set $\Gamma
:= \set{ f \in C }$ and denote $\gamma = \mu(\Gamma)$.

We now choose $t_0=0 < t_1 < t_2 < \ldots < 1$, so that denoting for
$i \geq 1$, $\Omega_i = \set{t_{i-1} \leq \Phi \leq t_i}$, and
setting $m_i = \mu(\Omega_i \setminus \Gamma)$, we have $m_i =
(b'-a'-\gamma) \alpha^{i-1}(1-\alpha)$, where $0 \leq \alpha \leq 1$
will be chosen later. Denote in addition $\Phi_i =
\brac{\frac{\Phi-t_{i-1}}{t_i - t_{i-1}} \vee 0} \wedge 1$,
$N_i = \sum_{j>i} m_j$. Applying H\"{o}lder's
inequality twice, we estimate:
\begin{eqnarray*}
\brac{\int_\Omega \abs{\nabla \Phi}^q d\mu}^{1/q} &=&
\brac{\sum_{i=1}^\infty \int_{\Omega_i \setminus \Gamma} \abs{\nabla
\Phi}^q d\mu}^{1/q} \geq \brac{\sum_{i=1}^\infty
m_i^{1-\frac{q}{q_0}} \brac{\int_{\Omega_i \setminus \Gamma}
\abs{\nabla \Phi}^{q_0}
d\mu}^{q/q_0}}^{1/q} \\
&\geq& \brac{\sum_{i=1}^\infty m_i^{1-\frac{q}{q_0}} (t_i-t_{i-1})^q
\brac{\int_{\Omega} \abs{\nabla \Phi_i}^{q_0} d\mu}^{q/q_0}}^{1/q} \\
&\geq& \brac{\sum_{i=1}^\infty m_i^{1-\frac{q}{q_0}}
(t_i-t_{i-1})^q Cap_{q_0}^q(\mu\set{\Phi_i = 1},1-\mu\set{\Phi_i = 0})}^{1/q} \\
&\geq& \sum_{i=1}^\infty (t_i-t_{i-1}) \brac{\sum_{i=1}^\infty
\frac{m_i^{1-p/p_0}}{Cap_{q_0}^p(\mu\set{\Phi \geq t_i},b)}}^{-1/p}
~.
\end{eqnarray*}
Since $\mu\set{\Phi \geq t_i} \geq a' + N_i$ and $Cap_{q_0}(s,b)$ is
non-decreasing in $s$, we continue to estimate as follows:
\begin{eqnarray*}
& & \brac{\frac{1}{\int_\Omega \abs{\nabla \Phi}^q d\mu}}^{p/q} \leq \sum_{i=1}^\infty
\frac{m_i^{1-p/p_0}}{Cap_{q_0}^p(\mu\set{\Phi \geq t_i},b)} \\
&\leq & \sum_{i=1}^\infty \frac{m_i^{1-p/p_0}}{m_{i+1}}
\int_{a'+N_{i+1}}^{a'+N_i} \frac{ds}{Cap_{q_0}^p(s,b)} \leq
\sum_{i=1}^\infty \frac{1}{\alpha m_i^{p/p_0}}
\int_{a'+N_{i+1}}^{a'+N_i} 
\frac{ds}{Cap_{q_0}^p(s,b)} \\
&\leq & \frac{1}{\alpha} \brac{\frac{\alpha}{1-\alpha}}^{p/p_0}  \sum_{i=1}^\infty
\int_{a'+N_{i+1}}^{a'+N_i} 
\frac{ds}{(s-a')^{p/p_0} Cap_{q_0}^p(s,b)} \\
&\leq& \frac{1}{\alpha} \brac{\frac{\alpha}{1-\alpha}}^{p/p_0}  \int_a^b \frac{ds}{(s-a)^{p/p_0}
Cap_{q_0}^p(s,b)} ~,
\end{eqnarray*}
where we have used that $m_{i+1} = \alpha m_i$, $m_i =
\frac{1-\alpha}{\alpha} N_i$, and in the last inequality the fact
that $Cap_{q_0}(s,b)$ is non-decreasing in $s$. The assertion now
follows by taking supremum on all $\Phi$ as above, and choosing the
optimal $\alpha = 1 - p/p_0$.
\end{proof}

\begin{lem} \label{lem:ugly-becomes-nice}
Let $1 \leq p \leq p_0 \leq \infty$, and let $N \in \J$ so that $N(t)^{1/\alpha} / t$ is
non-decreasing for some $\alpha>0$ (in particular this holds with $\alpha=1$ when $N$ is a Young
function). Then for any $t>0$:
\[
\brac{\int_t^\infty \frac{ds}{(s-t)^{p/p_0} N^\wedge(s)^p}}^{1/p} \leq \delta_{p,p_0,\alpha}
\brac{\int_t^\infty \frac{ds}{s^{p/p_0} N^\wedge(s)^p}}^{1/p} ~,
\]
where:
\begin{equation} \label{eq:delta}
\delta_{p,p_0,\alpha} \leq c 2^{1/\alpha} (1 - p/p_0)^{-1/p} ~,
\end{equation}
and $c>0$ is a universal constant.
\end{lem}
\begin{proof}
Let us evaluate the integral on $[t,2t]$ and $[2t,\infty)$
separately:
\[
 \int_{2t}^\infty \frac{ds}{(s-t)^{p/p_0} N^\wedge(s)^p} \leq  2^{p/p_0} \int_{2t}^\infty
\frac{ds}{s^{p/p_0} N^\wedge(s)^p} ~.
\]
On the other hand, since $N^\wedge(t^\alpha)/t$ is non-increasing:
\begin{eqnarray*}
 \int_{t}^{2t} \frac{ds}{(s-t)^{p/p_0} N^\wedge(s)^p} &\leq& \frac{t^{1-p/p_0}}{(1-p/p_0)
N^\wedge(t)^p} \leq \frac{t^{1-p/p_0} 2^{p/\alpha} }{(1-p/p_0) N^\wedge(2t)^p} \\
&\leq& \frac{2^{p/\alpha}}{2^{1-p/p_0}-1} \int_t^{2t} \frac{ds}{s^{p/p_0} N^\wedge(s)^p} ~.
\end{eqnarray*}
Summing these two expressions, the assertion follows.
\end{proof}

\begin{rem} \label{rem:gamma-delta}
We do not optimize on the dependence on $\alpha$ here, since in our applications $\alpha \geq 1$. In
this case, note that $\gamma_{p,p_0}$ in (\ref{eq:gamma}) and $\delta_{p,p_0,\alpha}$ in
(\ref{eq:delta}) conveniently satisfy:
\[
\gamma_{p,p_0} \delta_{p,p_0,\alpha} \leq C ~,
\]
where $C>0$ is some universal constant. This will be used in the proof of Theorem
\ref{thm:Orlicz-12-strong} below.
\end{rem}

\begin{lem} \label{lem:big-lemma}
Let $p_1,p_2,p_3 \in [1,\infty]$, and let $N_1 \in \J$ satisfy:
\[
\frac{N_1(t)^{1/p_3 + 1/p_2 - 1/p_1}}{t} \text{ is non-decreasing}
\]
and:
\begin{equation} \label{eq:integrability}
\int_1^\infty \frac{ds}{s^{p_2/p_1} N_1^\wedge(s)^{p_2}} < \infty ~,~ \int_0^1 \frac{ds}{s^{p_2/p_1}
N_1^\wedge(s)^{p_2}} = \infty ~.
\end{equation}
Let $N_2 : \Real_+ \rightarrow \Real_+$ be the function so that:
\begin{equation} \label{eq:def-N2}
 N_2^\wedge(t) := \frac{1}{\brac{\int_t^{\infty} \frac{ds}{s^{p_2/p_1}
N_1^\wedge(s)^{p_2}}}^{1/p_2}}
\end{equation}
Then:
\begin{enumerate}
\item
$N_2(t)^{1/p_3}/t$ is non-decreasing.
\item
If $p_2 \leq p_3$ then $N_2$ is a convex (hence Young) function.
\end{enumerate}
\end{lem}

\begin{proof}
Note that since $N_1 \in \J$, it is almost everywhere differentiable. Also note
that our integrability conditions (\ref{eq:integrability}) together with $N_1 \in \J$
ensure that $N_2 \in \J$. We will assume that $p_2 < \infty$,
the case $p_2 = \infty$ follows by taking limit.

For the first part, it is equivalent to show that $N_2^\wedge(t^{p_3})/t$ is
non-increasing, which in turn is equivalent to checking that $F(t)$, defined below, is
non-decreasing:
\[
F(t):= \int_t^{\infty} \frac{t^{p_2/p_3} ds}{s^{{p_2}/{p_1}} N_1^\wedge(s)^{p_2} } ~.
\]
Indeed:
\[
 G(t) := \frac{F'(t)}{t^{p_2/p_3-1}} = \frac{p_2}{p_3} \int_t^{\infty}
\frac{ds}{s^{{p_2}/{p_1}} N_1^\wedge(s)^{p_2}} - \frac{t}{t^{{p_2}/{p_1}}
N_1^\wedge(t)^{p_2} } ~,
\]
and the integrability condition (\ref{eq:integrability}) ensures that $\limsup_{t \rightarrow
\infty} G(t) = 0$. We will show that $G(t)$ is non-increasing, from which it will follow that
$G(t) \geq 0$, hence $F'(t) \geq 0$, as claimed. Indeed, for almost all $t>0$:
\begin{eqnarray*}
 G'(t) &=& -\frac{p_2/p_3}{t^{{p_2}/{p_1}} N_1^\wedge(t)^{p_2}} - \frac{(1-{p_2}/{p_1})
t^{-{p_2}/{p_1}} N_1^\wedge(t)^{p_2} -
t^{1-{p_2}/{p_1}} p_2 N_1^\wedge(t)^{p_2-1} (N_1^\wedge)'(t)}{N_1^\wedge(t)^{2p_2}} \\
&=& \frac{p_2}{t^{p_2/p_1} N_1^\wedge(t)^{p_2+1}} (t (N_1^\wedge)'(t) - (1/p_3 + 1/p_2 - 1/p_1)
N_1^\wedge(t)) ~.
\end{eqnarray*}
The last expression is indeed non-positive, since $N_1(t)^{1/p_3 + 1/p_2 - 1/p_1}/t$ is
non-decreasing, hence $N_1^\wedge(t)^{1/(1/p_3 + 1/p_2 - 1/p_1)}/t$ is non-increasing, and
by differentiating the latter expression one verifies that $(N_1^\wedge)'(t) \leq (1/p_3 + 1/p_2 -
1/p_1) N(t)/t$.

For the second part, let us substitute the definitions of $N_1^\wedge,N_2^\wedge$ in
(\ref{eq:def-N2}) and perform the change of variables $z = 1/s$. This amounts to:
\[
 N_2^{-1}(t)^{p_2} = \int_0^t \frac{N_1^{-1}(z)^{p_2}}{z^{2-p_2/p_1}} dz ~.
\]
Taking the derivative, we obtain that for almost every $t>0$:
\[
 \frac{p_2 N_2^{-1}(t)^{p_2-1}}{N_2'(N_2^{-1}(t))} = \frac{N_1^{-1}(t)^{p_2}}{t^{2-p_2/p_1}} ~.
\]
Multiplying by the denominator on the left hand side and taking the derivative once again yields
that for almost every $t>0$:
\[
 \frac{p_2(p_2-1) N_2^{-1}(t)^{p_2-2}}{N_2'(N_2^{-1}(t))} =  T(t) N_2'(N_2^{-1}(t)) +
\frac{N_1^{-1}(t)^{p_2}}{t^{2-{p_2}/{p_1}}} \frac{N_2''(N_2^{-1}(t))}{N_2'(N_2^{-1}(t))}
\]
with:
\[
T(t) = \frac{p_2 N_1^{-1}(t)^{p_2-1}}{N_1'(N_1^{-1}( t)) t^{2-{p_2}/{p_1}}} -
(2-{p_2}/{p_1})\frac{N_1^{-1}(t)^{p_2}}{t^{3-{p_2}/{p_1}}}
\]
In particular, $N_2$ is twice differentiable for almost every $t>0$, and it is clear that
$N_2'' \geq 0$ if $T \leq 0$ almost everywhere. The latter amounts to checking that for almost all
$z>0$:
\[
(2/p_2 - 1/p_1) N_1'(z) \geq N_1(z)/z ~.
\]
When $p_2 \leq p_3$, this follows from the stronger statement:
\[
(1/p_3 + 1/p_2 - 1/p_1) N_1'(z) \geq N_1(z)/z ~,
\]
which indeed holds for almost all $z>0$, as verified by differentiating
$\frac{N_1(z)^{1/p_3 + 1/p_2 - 1/p_1}}{z}$, which by assumption is non-decreasing.

\end{proof}

\subsection{Orlicz-Sobolev implies Isoperimetry for $q \geq 1$}

We can now prove the following extension of Theorem \ref{thm:Orlicz-12}:

\begin{thm} \label{thm:Orlicz-12-strong}
The assumption that $q\geq 2$ in Theorem \ref{thm:Orlicz-12} can be relaxed to $q \geq 1$, and the
assumption that $N$ is a Young function omitted, if we assume in addition that:
\[
\frac{N(t)^{1/q}}{t} \text{ is non-decreasing} ~.
\]
In this case, under our convexity assumptions, (\ref{eq:Orlicz-inq}) implies
(\ref{eq:Orlicz-12-conclusion}) with:
\begin{equation} \label{eq:CNq}
 C_{N,q} \geq c \inf_{0<t<1/2} \frac{t^{1/r - 1/p}}{\brac{\int_t^{\infty}
\frac{N^\wedge(t)^r ds}{s^{r/p} N^\wedge(s)^r  }}^{1/r}} ~,
\end{equation}
where $c>0$ is a universal constant, and:
\begin{equation} \label{eq:r-p}
p = \begin{cases} q^* & q < 2 \\ q  & q \geq 2 \end{cases} \quad r = \begin{cases} 2 & q < 2 \\ q  &
q \geq 2 \end{cases} ~.
\end{equation}
\end{thm}
\begin{rem} \label{rem:bounded-from-above}
Estimating the expression in (\ref{eq:CNq}) is connected to Hardy-type inequalities. We do not
proceed in this direction in this work, since for our applications the bounds are easy to deduce
directly. We remark that whenever $N(t)^\alpha/t$ is non-decreasing for some $\alpha > 0$,
$N^\wedge(t)^{1/\alpha}/t$ is non-increasing, and so:
\[
\frac{N^\wedge(t)}{N^\wedge(s)} \geq \brac{\frac{t}{s}}^{\alpha} \;\;\;\; \forall s \geq t ~.
\]
Using this estimate, it is immediate to show that the expression on the right
hand side of (\ref{eq:CNq}) is bounded from above by a universal constant whenever $1/q \leq \alpha \leq 1$, even if the infimum in (\ref{eq:CNq}) is replaced by a supremum. In particular, this obviously applies to all Young functions $N$ (with $\alpha=1$).
\end{rem}

\begin{proof}[Proof of Theorem \ref{thm:Orlicz-12-strong}]
First, note that whichever the value of $q$, we have:
\[
\frac{r}{p} + \frac{r}{q} = 2 ~.
\]
By Corollary \ref{cor:N-at-0}, we can always assume that $N(t) =
2 (t/N^{-1}(2))^q$ on $t \in [0,N^{-1}(2)]$, so that $N^\wedge(t) =
\frac{2^{1/q}}{N^{-1}(2)} t^{1/q}$ on $t\in[1/2,\infty)$, and therefore:
\begin{equation} \label{eq:integ1}
\int_1^\infty \frac{ds}{s^{r/p} N^\wedge(s)^{r}} < \infty ~.
\end{equation}
Using the assumption that $N(t)^{1/q}/t$ is non-decreasing, hence $N^\wedge(t^q)/t$ is non-increasing, it follows
that:
\begin{equation} \label{eq:integ2}
\int_0^1 \frac{ds}{s^{r/p} N^\wedge(s)^{r}} = \infty ~.
\end{equation}
The assumption (\ref{eq:Orlicz-inq}) implies by Proposition
\ref{prop:Capq-Lq-weak} that:
\begin{equation} \label{eq:cap-q-assumption}
 Cap_q(t,1/2) \geq D N^\wedge(t) \;\;\; \forall t \in [0,1/2] ~.
\end{equation}

We start with the case $q < 2$. Using Proposition \ref{prop:increase-Orlicz-q} (with $q_0=q,q=2$)
to pass from $Cap_q$ to $Cap_2$,
together with Lemma \ref{lem:ugly-becomes-nice} (with $\alpha = q$) and Remark
\ref{rem:gamma-delta}, we obtain that:
\[
 Cap_2(t,1/2) \geq c D N_2^\wedge(t) \;\;\; \forall t \in [0,1/2] ~,
\]
for some universal constant $c>0$, where $N_2$ is a function so that:
\[
 N_2^\wedge(t) := \frac{1}{\brac{\int_t^{\infty} \frac{ds}{s^{2/p} N^\wedge(s)^2}}^{1/2}} .
\]
Since $N(t)^{1/q}/t$ is non-decreasing and the integrability
conditions (\ref{eq:integ1}), (\ref{eq:integ2}) are fulfilled, we
can apply Lemma \ref{lem:big-lemma} with $N_1 = N$, $p_1 = p, p_2 =
2, p_3=2$, and conclude that $N_2$ is a Young function and that
$N_2(t)^{1/2}/t$ is non-decreasing.
Proposition \ref{prop:Capq-Lq} then implies that:
\begin{equation} \label{eq:Orlicz-inq-2}
 \forall f \in \F \;\; \frac{c}{4} D \norm{f - M_\mu f}_{N_2(\mu)} \leq \norm{ \abs{\nabla f}
}_{L_2(\mu)} ~.
\end{equation}
We can now apply Theorem \ref{thm:Orlicz-12}, and conclude that:
\[
 I(t) \geq c' D t^{1/2} N_2^\wedge(t) \;\;\; \forall t \in [0,1/2] ~.
\]
with $c'>0$ a universal constant. The value of $C_{N,q}$ in (\ref{eq:CNq}) ensures that this
implies:
\[
 I(t) \geq C_{N,q} D t^{1-1/q} N^\wedge(t) \;\;\; \forall t \in [0,1/2],
\]
as required. This concludes the proof when $q<2$.

When $q \geq 2$, we use a similar argument. Let $N_q$ denote the function so that:
\[
 N_q^\wedge(t) := \frac{1}{\brac{\int_t^{\infty} \frac{ds}{s N^\wedge(s)^q}}^{1/q}} .
\]
Again, by Lemma \ref{lem:big-lemma} with $N_1 = N$, $p_1 = p_2 = p_3
= q$, we know that $N_q$ is a Young function and that
$N_q(t)^{1/q}/t$ is non-decreasing.
Recalling Remark \ref{rem:bounded-from-above}, the assumption (\ref{eq:cap-q-assumption}) implies that:
\[
 Cap_q(t,1/2) \geq c D N_q^\wedge(t) \;\;\; \forall t \in [0,1/2] ~,
\]
for some universal $c>0$. Proposition \ref{prop:Capq-Lq} then implies that:
\[
 \forall f \in \F \;\; \frac{c}{4} D \norm{f - M_\mu f}_{N_q(\mu)} \leq \norm{ \abs{\nabla f}
}_{L_q(\mu)} ~.
\]
We can now apply Theorem \ref{thm:Orlicz-12}, and using the definition of $C_{N,q}$ in (\ref{eq:CNq}), conclude that:
\[
 I(t) \geq c' D t^{1-1/q} N_q^\wedge(t) \geq c'' C_{N,q} D t^{1-1/q} N^\wedge(t)\;\;\;
\forall t \in [0,1/2] ~,
\]
as required.
\end{proof}

\begin{cor} \label{cor:Orlicz-12-strong-easy}
Let $1 \leq q < \infty$, $N \in \J$, and assume that:
\[
N(t)^{1/q}/t \text{ is non-decreasing}~,~ \exists \alpha > 1/r-1/p \;\; N(t^{\alpha})/t \text{ is
non-increasing },
\]
with $r,p$ as in $(\ref{eq:r-p})$.
Then under our convexity assumptions, the assumption
(\ref{eq:Orlicz-inq}) implies the conclusion
(\ref{eq:Orlicz-12-conclusion}) with:
\begin{equation} \label{eq:CNq-specialized}
C_{N,q} \geq c (\alpha +1/p - 1/r)^{1/r} ~,
\end{equation}
where $c>0$ is a universal constant.
\end{cor}
\begin{proof}
The assumptions imply that:
\begin{equation} \label{eq:N-alpha-decay}
 \brac{\frac{t}{s}}^{1/q} \leq \frac{N^\wedge(t)}{N^\wedge(s)} \leq \brac{\frac{t}{s}}^{\alpha}
\end{equation}
whenever $s \geq t$.
Applying Theorem \ref{thm:Orlicz-12-strong} and using
(\ref{eq:N-alpha-decay}), it is straightforward to obtain a lower
bound on the expression in (\ref{eq:CNq}), which yields the bound in
(\ref{eq:CNq-specialized}).
\end{proof}

It was shown by Bobkov and Zegarlinski \cite[Proposition 3.1]{BobkovZegarlinski} (generalizing the
case $q=2$ due to Bobkov and G\"{o}tze \cite[Proposition 4.1]{BobkovGotzeLogSobolev}) that the
following $q$-log-Sobolev inequality (with $1 \leq q \leq 2$):
\begin{equation} \label{eq:q-log-Sob}
\forall f \in \F \;\;\; D_1 \brac{ \int |f|^q \log |f|^q d\mu - \int |f|^q d\mu \log (\int |f|^q d\mu)}^{1/q}  \leq \norm{\abs{\nabla f}}_{L_q(\mu)}
\end{equation}
is equivalent to the inequality:
\begin{equation} \label{eq:N_q-log-Sob}
\forall f \in \F \;\;\; D_2 \norm{ f - E_\mu f }_{\varphi_q(\mu)} \leq \norm{\abs{\nabla
f}}_{L_q(\mu)},
\end{equation}
where $\varphi_q(t) = t^q \log(1+t^q)$, and $D_1 \simeq D_2$ uniformly on $q \in [1,2]$.
Using Lemma \ref{lem:E-M}, we can replace $E_\mu f$ in
(\ref{eq:N_q-log-Sob}) by $M_\mu f$, at the expense of an additional
universal constant. Using Corollary \ref{cor:Orlicz-12-strong-easy}
with $N=\varphi_q$ and $\alpha = \frac{1}{2q} > 1/q - 1/2$ in the
range $q \in (1,2]$, we can easily show that the $q$-log-Sobolev
inequality (\ref{eq:q-log-Sob}) implies a corresponding
isoperimetric inequality. However, to handle the entire range $q \in
[1,2]$ uniformly, we will need to turn to Theorem
\ref{thm:Orlicz-12-strong}.

\begin{cor} \label{cor:q-log-Sob-implies-isop}
Under our convexity assumptions, the $q$-log-Sobolev inequality (\ref{eq:q-log-Sob}) for $1
\leq q \leq 2$ implies the following isoperimetric inequality:
\begin{equation} \label{eq:q-log-Sob-Iso}
 \tilde{I}(t) \geq c D_1 t \log^{1/q} 1/t \;\;\; \forall t \in [0,1/2] ~,
\end{equation}
where $c>0$ is a universal constant.
\end{cor}

\begin{proof}
$\varphi_q(t)^{1/q}/t$ is non-decreasing, so using Lemma
\ref{lem:E-M} and Corollary \ref{cor:N-at-0}, (\ref{eq:q-log-Sob})
implies that:
\[
\norm{ f - M_\mu f }_{N_q(\mu)} \leq C D_1 \norm{\abs{\nabla
f}}_{L_q(\mu)},
\]
where $C>0$ is a universal constant and:
\[
N_q(t) = \begin{cases} 2 (t/\varphi_q^{-1}(2))^q & t \in [0,\varphi_q^{-1}(2)] \\
\varphi_q(t) & t \in [\varphi_q^{-1}(2),\infty) \end{cases} ~.
\]
Note that $N_q(t)^{1/q}/t$ is still non-decreasing. Clearly:
\[
N_q^\wedge(t) = \begin{cases} \varphi_q^\wedge(t) & t \in [0,1/2] \\
\frac{2^{1/q}}{\varphi_q^{-1}(2)} t^{1/q} & t \in [1/2,\infty)
\end{cases} ~,
\]
and a standard calculation shows that:
\begin{equation} \label{eq:N_q-wedge}
\varphi_q^\wedge(t) \simeq t^{1/q} \log^{1/q}(1+1/t) \;\;\; \forall
t \in [0,1/2] ~,
\end{equation}
uniformly on $q\in [1,2]$. Hence, using Theorem
\ref{thm:Orlicz-12-strong} to deduce the isoperimetric inequality
(\ref{eq:q-log-Sob-Iso}), it remains to bound the expression in
(\ref{eq:CNq}) from below uniformly in $q \in [1,2]$. This amounts
to showing that:
\[
\sup_{0<t<1/2} \int_t^{\infty} \frac{t^{2/p-1} N_q^\wedge(t)^2
ds}{s^{2/p} N_q^\wedge(s)^2} \leq C_1 ~,
\]
where $p=q^*$ and $C_1>0$ is a universal constant. First, we bound
the tail of this integral using (\ref{eq:N_q-wedge}):
\[
 \int_{1/2}^{\infty} \frac{t^{2/p-1} N_q^\wedge(t)^2
ds}{s^{2/p} N_q^\wedge(s)^2} \leq t^{2/p-1} \varphi_q^\wedge(t)^2
 \int_{1/2}^\infty \frac{\varphi_q^{-1}(2)^2 ds}{2^{2/q} s^2} \leq C_2
 t \log^{2/q}(1+1/t) ~,
\]
which is bounded by a universal constant for $t \in [0,1/2]$. Next,
we use (\ref{eq:N_q-wedge}) and the change of variables $v =
\log(1+1/s)$ to bound:
\begin{eqnarray*}
\int_t^{1/2} \frac{t^{2/p-1} N_q^\wedge(t)^2 ds}{s^{2/p}
N_q^\wedge(s)^2} &\leq& C_3 \int_t^{1/2} \frac{t \log^{2/q}(1+1/t)
ds}{s^2 \log^{2/q}(1+1/s)} \leq C_3 t \log^{2/q}(1+1/t) \int_{\log
3}^{\log(1+1/t)} \frac{\exp(v)}{v^{2/q}} dv  \\
& \leq & C_4 t \log^{2/q}(1+1/t)
\frac{\exp(\log(1+1/t))}{\log^{2/q}(1+1/t)} = C_4 (1 + t) ~.
\end{eqnarray*}
We see that this is also bounded in the range $t \in [0,1/2]$, and
this concludes the proof.

\end{proof}
\begin{rem} \label{rem:BZ-direction}
The case $q=2$ was previously shown by Bakry--Ledoux \cite{BakryLedoux}
and Ledoux \cite{LedouxSpectralGapAndGeometry}. For general $1 \leq q \leq 2$, the reverse direction without any convexity assumptions was shown by Bobkov and Zegarlinski \cite{BobkovZegarlinski}, and given a different proof by Sodin and the author \cite{EMilmanSodinIsoperimetryForULC}. We will see a general argument for this in the next theorem.
\end{rem}

\subsection{Isoperimetry implies Orlicz-Sobolev}

\begin{thm} \label{thm:Orlicz-21-strong}
Let $1 \leq q < \infty$, and set $p = q^*$. Let $N \in \J$, so that
$N(t)^{1/q}/t$ is non-decreasing. Then:
\begin{equation} \label{eq:I-implies-Orlicz-assumption}
 \tilde{I}(t) \geq D t^{1-1/q} N^\wedge(t) \;\; \forall t \in [0,1/2]
\end{equation}
implies:
\begin{equation} \label{eq:I-implies-Orlicz-conclusion}
\forall f \in \F \;\; B_{N,q} D \norm{f - M_\mu f}_{N(\mu)} \leq \norm{ \abs{\nabla f}}_{L_q(\mu)}
~,
\end{equation}
where:
\begin{equation} \label{eq:B-estimate}
 B_{N,q} \geq \frac{1}{4} \inf_{0<t<1/2} \frac{1}{\brac{\int_t^{1/2}
\frac{N^\wedge(t)^p ds}{s N^\wedge(s)^p  }}^{1/p}} ~.
\end{equation}
\end{thm}
\begin{proof}
We rewrite (\ref{eq:I-implies-Orlicz-assumption}) using Corollary \ref{cor:cap1} as:
\[
 Cap_1(t,1/2) \geq D t^{1/p} N^\wedge(t) \;\;\; \forall t \in [0,1/2] ~.
\]
Using Proposition \ref{prop:increase-Orlicz-q} (with $q_0=1,q=q$) to pass from $Cap_1$ to $Cap_q$,
we obtain that:
\[
 Cap_q(t,1/2) \geq D G_p(t) \;\;\; \forall t \in [0,1/2] ~,
\]
where $G_p$ is defined on $[0,1/2]$ as:
\[
G_p(t) := \frac{1}{\brac{\int_t^{1/2} \frac{ds}{s N^\wedge(s)^p}}^{1/p}} ~.
\]
Incidentally, if we replace $1/2$ in the upper range of the above integral by $\infty$, by Lemma
\ref{lem:big-lemma} with $N_1 = N$, $p_1 = p, p_2 = p, p_3=q$, we would have that
$G_p(t^q)/t$ is non-increasing, but this will not be used. The estimate in (\ref{eq:B-estimate})
ensures that:
\[
 G_p(t) \geq 4 B_{N,q} N^\wedge(t) \;\;\; \forall t \in [0,1/2] ~,
\]
so we know that:
\begin{equation} \label{eq:I-implies-Orlicz-almost}
 Cap_q(t,1/2) \geq 4 B_{N,q} D N^\wedge(t) \;\;\; \forall t \in [0,1/2] ~.
\end{equation}
Using that $N(t)^{1/q}/t$ is non-decreasing,
Proposition \ref{prop:Capq-Lq} then implies (\ref{eq:I-implies-Orlicz-conclusion}), as asserted.
\end{proof}

\begin{rem} \label{rem:Orlicz-21-btw}
Note that the assumption (\ref{eq:I-implies-Orlicz-assumption}) implies
(\ref{eq:I-implies-Orlicz-almost}) without assuming that $N(t)^{1/q}/t$ is non-decreasing.
\end{rem}

\begin{cor} \label{cor:Orlicz-21-strong-easy}
Let $1 \leq q < \infty$ and set $p=q^*$. Assume that:
\[
N(t)^{1/q}/t \text{ is non-decreasing and } N(t^{\alpha})/t \text{ is non-increasing },
\]
with some $\alpha > 0$. Then the assumption (\ref{eq:I-implies-Orlicz-assumption}) implies the
conclusion (\ref{eq:I-implies-Orlicz-conclusion})
with:
\begin{equation} \label{eq:BNq-specialized}
B_{N,q} \geq c \alpha^{1/p} ~,
\end{equation}
where $c>0$ is a universal constant.
\end{cor}
\begin{proof}
Exactly as in the proof of Corollary \ref{cor:Orlicz-12-strong-easy}.
\end{proof}

Using this for $N=\varphi_q$ and $\alpha = \frac{1}{2q}$, we see that as already noted in Remark \ref{rem:BZ-direction}, the isoperimetric inequality $(\ref{eq:q-log-Sob-Iso})$ implies without any further assumptions the $q$-log-Sobolev inequalities (\ref{eq:N_q-log-Sob}) and (\ref{eq:q-log-Sob}).

\subsection{Summary}

To conclude this section, we provide a slightly stronger version of Theorem \ref{thm:Intro-Orlicz}
from the Introduction, on the equivalence of isoperimetric and
Orlicz-Sobolev functional inequalities under our convexity assumptions. Our results in this
section are more general, but this theorem summarizes the most useful cases given by Theorem
\ref{thm:Orlicz-12} and Corollaries \ref{cor:Orlicz-12-strong-easy} and
\ref{cor:Orlicz-21-strong-easy}, and generalizes the results from \cite{EMilmanRoleOfConvexityArxiv} (which dealt with the case $N(t) = t^p$ below).

\begin{thm} \label{thm:Orlicz-summary}
Let $1 \leq q \leq \infty$, and let $N \in \J$. Assume that:
\begin{equation} \label{eq:OS-A}
\text{our convexity assumptions are satisfied} ~,
\end{equation}
\begin{equation} \label{eq:OS-B}
\text{$N$ is a Young function} ~,
\end{equation}
\begin{equation} \label{eq:OS-C}
N(t)^{1/q}/t \text{ is non-decreasing } ~, ~ \exists \alpha > 0 \;\; N(t^\alpha)/t
\text{ is non-increasing} ~,
\end{equation}
\begin{equation} \label{eq:OS-D}
 \alpha > 1/q - 1/2 ~.
\end{equation}
Then the following statements are equivalent:
\begin{enumerate}
 \item
\[
 \forall f \in \F \;\; D_1 \norm{f - M_\mu f}_{N(\mu)} \leq \norm{ \abs{\nabla f} }_{L_q(\mu)}
\]
\item
\[
 \tilde{I}(t) \geq D_2 t^{1-1/q} N^\wedge(t)
\;\;\; \forall t \in [0,1/2] ~,
\]
\end{enumerate}
where the best constants $D_1,D_2$ above satisfy:
\[
c_1 C_{\alpha,q} D_1 \leq D_2 \leq c_2 B_{\alpha,q} D_1 ~,
\]
with $c_1,c_2 > 0$ universal constants and:
\begin{equation} \label{eq:C-B-defs}
 C_{\alpha,q} = \begin{cases} (\alpha + 1/2 - 1/q)^{1/2} & q<2 \\ \alpha^{1/q} & q \geq 2
\end{cases} \quad
, \quad B_{\alpha,q} = \alpha^{1/q - 1} ~.
\end{equation}
In fact, among the assumptions (\ref{eq:OS-A}), (\ref{eq:OS-B}), (\ref{eq:OS-C}),
(\ref{eq:OS-D}):
\begin{itemize}
 \item For the direction $(2) \Rightarrow (1)$ only (\ref{eq:OS-C}) is needed.
\item For the direction $(1) \Rightarrow (2)$ with $q \geq 2$ only (\ref{eq:OS-A}) and
one of (\ref{eq:OS-B}) or (\ref{eq:OS-C}) are needed, and if (\ref{eq:OS-B}) is used then
$C_{\alpha,q}$ can be chosen to be $1$.
\item  For the direction $(1) \Rightarrow (2)$ with $q < 2$ (\ref{eq:OS-B}) is not needed.
\end{itemize}
\end{thm}

\begin{proof}
The direction $(1) \Rightarrow (2)$ was proved in Theorem \ref{thm:Orlicz-12} and Corollary
\ref{cor:Orlicz-12-strong-easy}.
The direction $(2) \Rightarrow (1)$ was proved in Corollary \ref{cor:Orlicz-21-strong-easy}.
\end{proof}

\section{Tensorization} \label{sec:tensorization}

As mentioned in the Introduction, the results of Section
\ref{sec:general} coupled with the results of Section
\ref{sec:capacities} on the equivalence of capacity inequalities and
Orlicz-Sobolev inequalities, allow us to directly infer
isoperimetric inequalities from capacity inequalities (under
convexity assumptions of course).

\begin{thm} \label{thm:cap-iso}
Let $1 \leq q < \infty$, and let $N\in \J$. If:
\begin{equation} \label{eq:as-A}
\text{our convexity assumptions are satisfied and $N(t)^{1/q}/t$ is
non-decreasing} ~,
\end{equation}
\begin{equation}\label{eq:as-B}
\text{$N$ is a Young function} ~,
\end{equation}
\begin{equation} \label{eq:as-C}
 \exists \alpha > 0 \;\; N(t^\alpha)/t \text{ is non-increasing } ~,
\end{equation}
\begin{equation} \label{eq:as-D}
 \alpha > 1/q - 1/2 ~.
\end{equation}
then the following statements are equivalent:
\begin{enumerate}
 \item
\[
 Cap_q(t,1/2) \geq D_1 N^\wedge(t) \;\;\; \forall t \in [0,1/2] ~,
\]
\item
\[
 I(t) \geq D_2 t^{1-1/q} N^\wedge(t) \;\;\; \forall t \in [0,1/2] ~,
\]
\end{enumerate}
where the best constants $D_1,D_2$ above satisfy:
\[
c_1 C_{\alpha,q} D_1 \leq D_2 \leq c_2 B_{\alpha,q} D_1 ~,
\]
with $c_1,c_2 > 0$ universal constants and $B_{\alpha,q}, C_{\alpha,q}$ as in (\ref{eq:C-B-defs}).
In fact, among the assumptions (\ref{eq:as-A}), (\ref{eq:as-B}), (\ref{eq:as-C}), (\ref{eq:as-D}):
\begin{itemize}
\item For the direction $(2) \Rightarrow (1)$ only (\ref{eq:as-C}) is needed.
\item For the direction $(1) \Rightarrow (2)$ with $q \geq 2$ only (\ref{eq:as-A}) and
one of (\ref{eq:as-B}) or (\ref{eq:as-C}) are needed, and if (\ref{eq:as-B}) is used then
$C_{\alpha,q}$ in (\ref{eq:C-B-defs}) can be chosen to be $1$.
\item  For the direction $(1) \Rightarrow (2)$ with $q < 2$ (\ref{eq:as-B}) is not needed.
\end{itemize}
\end{thm}
\begin{proof}
The direction $(1) \Rightarrow (2)$ follows from Proposition \ref{prop:Capq-Lq} coupled with
Theorem \ref{thm:Orlicz-12} and Corollary \ref{cor:Orlicz-12-strong-easy}. The direction $(2)
\Rightarrow (1)$ follows from Theorem \ref{thm:Orlicz-21-strong}, Remark \ref{rem:Orlicz-21-btw} and
Corollary \ref{cor:Orlicz-21-strong-easy} (note that we indeed do not need the assumption that
$N(t)^{1/q}/t$ is non-decreasing).
\end{proof}

It has been established in recent years 
that several other types of functional inequalities are equivalent
to $2$-capacity inequalities. These include Beckner-type
inequalities \cite{BCRHard} (including the Lata{\l}a--Oleszkiewicz
inequality as in \cite{LatalaOleszkiewiczBecknerInq}) and additive
$\Phi$-Sobolev inequalities \cite{RobertoZegarlinski}. The advantage
of these inequalities compared to the Orlicz-Sobolev inequalities
lies in the fact that they admit tensorization. This easily allows
us to deduce an extension of the dimension-free tensorization
results of Barthe--Cattiaux--Roberto \cite{BCRHard,BCRSoft}. We
demonstrate this with Beckner-type inequalities (\ref{eq:Beckner}),
using Theorem 9 and Lemma 8 in \cite{BCRHard} as cited in
\cite{RobertoZegarlinski} (with a trivial change of notation):

\begin{thm}[Barthe--Cattiaux--Roberto] \label{thm:cap2}
Let $T:[1,\infty) \rightarrow \Real_+$ denote a non-decreasing function such that $T(t)^2/t$ is
non-increasing. Assume that our metric space $(\Omega,d)$ is Euclidean space
$(\Real^n,\abs{\cdot})$ and that $\mu$ is an absolutely continuous probability measure on
$\Real^n$. Then the following statements are equivalent:
\begin{enumerate}
 \item
\begin{equation} \label{eq:Beckner}
 \forall f \in \F \;\;\; D_1 \sup_{p \in (1,2)} \brac{ \int f^2 d\mu - (\int |f|^p d\mu)^{2/p}
}^{1/2} T\brac{\frac{1}{2-p}} \leq \norm{ \abs{\nabla f} }_{L_2(\mu)} ~,
\end{equation}
\item
\[
 Cap_2(t,1/2) \geq D_2 t^{1/2} T(\log(1+1/t)) \;\;\; \forall t \in [0,1/2] ~,
\]
\end{enumerate}
with the best constants $D_1,D_2$ above satisfying $D_1 / \sqrt{6} \leq D_2 \leq \sqrt{20} D_1$.
\end{thm}

It is known (e.g. \cite{BCRHard}) that Beckner-type inequalities (\ref{eq:Beckner}) admit
tensorization, in the sense that if they hold for $(M,g,\mu)$ then they also hold for the Riemannian
product $(M^{\times k},g^{\otimes k},\mu^{\otimes k})$ for any $k\geq 1$. To obtain the most
general result, we will also need the following remarkable observation of Franck Barthe
\cite[Theorem 10]{BartheTensorizationGAFA} (which in fact holds for very general metric probability
spaces, but for simplicity we quote it in less general form; see also Ros
\cite{RosIsoperimetricProblemNotes}):

\begin{thm}[Barthe] \label{thm:Barthe}
Let $(M,g)$ denote a Riemannian manifold equipped with an absolutely continuous Borel
probability measure $\mu$, and let $\nu$ denote an even log-concave probability measure with
continuous density on $\Real$.  If:
\[
 I_{(M,g,\mu)} \geq I_{(\Real,\abs{\cdot},\nu)}
\]
then for any $k\geq 1$:
\[
I_{(M^{\times k},g^{\otimes k},\mu^{\otimes k})} \geq I_{(\Real^k,\abs{\cdot}^{\otimes
k},\nu^{\otimes k})} ~.
\]
\end{thm}

We can now state the following extension of the tensorization results of
Barthe--Cattiaux--Roberto \cite{BCRHard,BCRSoft}, already roughly stated in the
Introduction. These authors obtained their result in \cite{BCRHard} primarily for the probability
measures $Z_\alpha \exp(-|x|^\alpha) dx$ on $\Real$ for $\alpha \in [1,2]$, and in \cite{BCRSoft}
for probability measures of the form $\exp(-\Phi(|x|)) dx$ where $\Phi$ is a convex function
such that $\sqrt{\Phi}$ is concave. These results cover the entire spectrum between
exponential and Gaussian tail decay. A result of Talagrand
\cite{TalagrandProductOfExponentialMeasure} asserts that for a dimension free
concentration inequality (and in particular, an isoperimetric inequality) to hold for an arbitrary
tensor power of a measure, it must have at least exponential tail decay. On the other hand, by the
Central-Limit Theorem (which applies by the fast tail decay), it is clear that an arbitrary tensor
power cannot have an isoperimetric inequality better than the Gaussian measure $\gamma$. This
explains the restriction on $\alpha$ and $\Phi$ above. In this sense the tensorization results of
\cite{BCRHard,BCRSoft} are sharp, but it would be interesting to obtain analogous results for
arbitrary log-concave measures whose isoperimetric profile is not better than that of the Gaussian
measure (due to the Central-Limit obstruction).

We will say that a function $J:[0,1]
\rightarrow \Real_+$ does not violate the Central-Limit obstruction if:
\begin{equation} \label{eq:obstruction0}
\exists D_0 < \infty \;\;\; \limsup_{t \rightarrow 0+} \frac{J(t)}{I_{(\Real,\abs{\cdot},\gamma)}(t)} \leq
D_0 ~.
\end{equation}
As pointed out to us by Sasha Sodin, even when $I_{(\Real,\abs{\cdot},\nu)}$ does not violate the obstruction, some badly behaved examples of (log-concave) measures $\nu$ have been constructed by Barthe \cite[Theorem 11]{BartheIsoperimetryUniformDistance}, which do not admit tensorization: $\liminf_{t \rightarrow 0+} \frac{I_{(\Real^2,\abs{\cdot},\nu^{\otimes 2})}}{I_{(\Real,\abs{\cdot},\nu)}} = 0$.
We will therefore need to impose some additional control over the isoperimetric profile. We will say that $J$ does not violate the Central-Limit obstruction \emph{with control rate $D$} if:
\begin{equation} \label{eq:obstruction}
\exists D < \infty \;\;\; 0<t \leq s \leq 1/2 \;\; \Rightarrow \;\;
\frac{J(t)}{I_{(\Real,\abs{\cdot},\gamma)}(t)} \leq D \frac{J(s)}{I_{(\Real,\abs{\cdot},\gamma)}(s)}
~.
\end{equation}
A function $f :[a,b] \rightarrow \Real_+$ for which there exists $C>1$ such that  $f(t) \leq C
f(s)$ for any $a \leq t \leq s \leq b$ ($a \leq s \leq t \leq b$) will be called \emph{essentially
non-decreasing (non-increasing)} (with constant $C$) on $[a,b]$.

\begin{thm} \label{thm:tensorization}
Let $J:[0,1] \rightarrow \Real_+$ denote an arbitrary continuous concave function vanishing at
$\set{0,1}$ and symmetric about the point $1/2$. Assume that (\ref{eq:obstruction}) holds, so
that the Central-Limit obstruction is not violated with control rate $D$. Let $(M,g)$ denote a
Riemannian manifold equipped with an absolutely continuous Borel probability measure $\mu$, and
assume that:
\begin{equation} \label{eq:tensor-assumption}
 I_{(M,g,\mu)}(t) \geq J(t) \;\;\; \forall t \in [0,1] ~.
\end{equation}
Then (without any additional convexity assumptions) there exists a constant
$c_D>0$ depending only on $D$, such that for any $k\geq 1$:
\[
I_{(M^{\times k},g^{\otimes k},\mu^{\otimes k})}(t) \geq c_D J(t) \;\;\; \forall t \in [0,1] ~.
\]
\end{thm}

The case $J(t)= c t$, $t\in[0,\frac{1}{2}]$, was settled by Bobkov
and Houdr\'e \cite{BobkovHoudre}, and should be interpreted as
stating that Cheeger's isoperimetric inequality is preserved (up to
a constant) under tensorization. The case $J(t) = c
I_{(\Real,\abs{\cdot},\gamma)}$ follows from the classical
isoperimetric inequality for the $k$-dimensional Gaussian measure
$\gamma_k$ due to Sudakov--Tsirelson \cite{SudakovTsirelson} and
independently Borell \cite{Borell-GaussianIsoperimetry}, stating
that $I_{(\Real^k,\abs{\cdot},\gamma_k)} =
I_{(\Real,\abs{\cdot},\gamma)}$, together with an application of
Barthe's Theorem \ref{thm:Barthe} (see also
\cite{BartheMaureyIsoperimetricInqs} for an alternative derivation).
Interpolating between these two extremes is the result of
Barthe--Cattiaux--Roberto in \cite{BCRHard}, who treated the case $J
= I_{(\Real,\abs{\cdot},\mu_\alpha)}$ where $d\mu_\alpha = Z_\alpha
\exp(-|x|^\alpha) dx$ and $\alpha \in [1,2]$, in which case $J(t)
\simeq t \log^{1-1/\alpha}(1+\frac{1}{t})$ uniformly on $t \in
[0,1/2]$ and $\alpha \in [1,2]$.
The more general case
when $J(t) = I_{(\Real,\abs{\cdot},\mu_\Phi)}$ and $d\mu_{\Phi} = \exp(-\Phi(|x|)) dx$ with $\Phi : \Real_+ \rightarrow \Real_+$ convex, so that $\sqrt{\Phi}$ is concave and $\Phi$ is $C^2$ at a neighborhood of infinity, was treated by these authors in \cite{BCRSoft}. One may show that in this case, $c_{1,\Phi} \leq J(t) / I_\Phi(t) \leq c_{2,\Phi}$ uniformly on $t \in [0,1/2]$, where $I_\Phi(t) = t \Phi' \circ \Phi^{-1}(\log (1+\frac{1}{t}))$ and $c_{i,\Phi}$ depend solely on $\Phi$. It is easy to check that (\ref{eq:obstruction}) is satisfied with $D \leq C$ in the first case and $D \leq C_\Phi$ in the second, where $C \geq 1$ is a universal constant and $C_\Phi \geq 1$ depends solely on $\Phi$ (in fact, if we replace $I_{(\Real,\abs{\cdot},\gamma)}$ in
(\ref{eq:obstruction}) by the equivalent $t \log^{1/2}(1+\frac{1}{t})$, then $\sqrt{\Phi}$ is concave iff $J=I_\Phi$ satisfies the modified (\ref{eq:obstruction}) with $D=1$).

Our formulation of Theorem \ref{thm:tensorization} using the condition (\ref{eq:obstruction}), without refering to an auxiliary profile $I_\nu$ where $\nu$ is some 1-dimensional density, seems more natural than previous requirements, and this will also be evident in the proof. As mentioned in the Introduction, it seems possible to produce a proof of this theorem using the approach of \cite{BCRSoft}, but the main obstacle would be to pass from the isoperimetric inequality $I(t) \geq J(t)$ to the appropriate $2$-capacity inequality, using only $J$ and without passing via the auxiliary density $\nu$ (compare with Theorem 7 and Propositions 9,13 in \cite{BCRSoft}), which would otherwise result in requiring some additional technical assumptions and in the constant $c_D$ to depend on $J$. On the other hand, without any further technical assumptions,
Theorem \ref{thm:tensorization} basically follows from the argument used to derive Theorem \ref{thm:cap-iso}
coupled with Theorems \ref{thm:cap2} and \ref{thm:Barthe}. To make this precise we will need to be slightly more careful.

\begin{proof}[Proof of Theorem \ref{thm:tensorization}]
By a result of Sergey Bobkov \cite{BobkovExtremalHalfSpaces}, the map $\nu \mapsto
I_{(\Real,\abs{\cdot},\nu)}$ is a one-to-one correspondence between even log-concave probability
measures $\nu$ with continuous density on $\Real$ and concave functions $J$ as in the theorem
(without the assumption (\ref{eq:obstruction})). Therefore, there exists a measure $\nu_0$ on
$\Real$ as above such that $I_{(\Real,\abs{\cdot},\nu_0)} = J$. We will prove the assertion for the
case that $(M,g,\mu)$ is $(\Real,\abs{\cdot},\nu_0)$, the general case will then follow from
Barthe's Theorem \ref{thm:Barthe}. We only use the fact that $\nu_0$ is log-concave, and work directly with its isoperimetric profile $J$. By approximating $J$ if necessary, we can always assume that $J$
is strictly increasing on $[0,1/2]$.

Let $I_0 : \Real_+ \rightarrow \Real_+$ be defined as:
\[
 I_0(t) = t \log^{1/2} (1+1/t) ~.
\]
As mentioned in the Introduction, it is known that $I_{(\Real,\abs{\cdot},\gamma)} \simeq I_0$ on $[0,1/2]$, so we may assume that (\ref{eq:obstruction}) holds with $I_{(\Real,\abs{\cdot},\gamma)}$ replaced by $I_0$.

Denote $g(t) = \min_{s \in [t,1/2]} J(s)/I_0(s)$ for $t \in [0,1/2]$. Clearly $g(1/2) =
J(1/2)/I_0(1/2)$, $g$ is non-decreasing and $g \leq J/I_0 \leq D g$ on $[0,1/2]$. Now denote:
\[
 J_1(t) = \begin{cases} J(t) & t \in [0,1/2] \\ 2 J(1/2) t & t \in [1/2,\infty) \end{cases} \quad ,
\quad
J_0(t) = \begin{cases} g(t) I_0(t) & t \in [0,1/2] \\ 2 J(1/2) t & t \in [1/2,\infty)
\end{cases} ~.
\]
Clearly $J_0, J_1 \in \J$ and $J_0 \leq J_1 \leq D J_0$. Since $I_0(t) / t^{1/2}$ increases on
$[0,1/2]$, 
$t/I_0(t)$ increases on $\Real_+$, and since $J$ is
concave and hence $J(t) / t$ is non-increasing on $[0,1/2]$, we have the following elementary facts:

\begin{enumerate}
\item  $J_0(t)/ t^{1/2}$ is increasing and $J_1(t)/ t^{1/2}$ is
essentially non-decreasing on $\Real_+$.
\item $J_1(t)/t$ is non-increasing and $J_0(t)/t$ is essentially non-increasing on
$\Real_+$.
\item $J_0(t)/I_0(t)$ is non-decreasing and $J_1(t)/I_0(t)$ is
essentially non-decreasing on $\Real_+$.
\end{enumerate}

Next, let $N \in \J$ denote the function so that:
\[
 N^\wedge(t) = \frac{1}{\brac{\int_t^\infty \frac{ds}{J_1(s)^2}}^{1/2}} ~.
\]
Indeed, Fact 3 implies that $N^\wedge(0) = 0$, and together with the linear growth of $J_1$ at
infinity, this means that  $N^\wedge \in \J$ and hence $N \in \J$. We now apply Lemma
\ref{lem:big-lemma} with $N_1 = J_1^\wedge$ and $p_1=\infty, p_2=p_3=2$. The appeal to Lemma
\ref{lem:big-lemma} is legitimate since $N^\wedge \in \J$ and since $J_1^\wedge(t)/t$ is
non-decreasing by Fact 2. We deduce that:
\begin{equation} \label{eq:N-prop}
\text{$N$ is a Young function and $N(t)^{1/2}/t$ is non-decreasing.}
\end{equation}
In addition, Facts 2 and 3 provide the following estimates:
\begin{equation} \label{eq:last-bounds}
 \frac{t}{s} \leq \frac{J_1(t)}{J_1(s)} \leq D \frac{I_0(t)}{I_0(s)} \;\;\; \forall \; 0 < t \leq s
< \infty ~,
\end{equation}
and an elementary computation provided in Lemma \ref{lem:last-thing} below implies that:
\begin{equation} \label{eq:last-thing}
 N^\wedge(t) \simeq_D \frac{J_1(t)}{\sqrt{t}} \simeq_D \frac{J_0(t)}{\sqrt{t}} 
\;\;\; \forall t \in [0,1] ~,
\end{equation}
with $\simeq_D$ meaning that the bounds depend on $D$.

Our assumption (\ref{eq:tensor-assumption}) on the space $(\Real,\abs{\cdot},\nu_0)$ implies by
Corollary \ref{cor:cap1} that:
\[
 Cap_1(t,1/2) \geq J(t) \;\;\; \forall t \in [0,1/2] ~.
\]
Proposition \ref{prop:increase-Orlicz-q} (with $q_0=1,q=q$)
implies that for all $t \in [0,1/2]$:
\begin{equation} \label{eq:cap2-mine}
 Cap_2(t,1/2) \geq \frac{1}{\brac{\int_t^{1/2} \frac{ds}{J(s)^2} }^{1/2}} \geq
\frac{1}{\brac{\int_t^{\infty} \frac{ds}{J_1(s)^2} }^{1/2}} = N^\wedge(t) ~.
\end{equation}

Now let $T : \Real_+ \rightarrow \Real_+$ denote the function satisfying:
\[
N^\wedge(t) = t^{1/2} T(\log(1+1/t)) ~,
\]
and rewrite (\ref{eq:cap2-mine}) as:
\begin{equation} \label{eq:cap2-BCR}
 Cap_2(t,1/2) \geq N^\wedge(t) = t^{1/2} T(\log(1+1/t)) \;\;\; \forall t \in [0,1/2]
\end{equation}
which holds for the space $(\Real,\abs{\cdot},\nu_0)$. By (\ref{eq:N-prop}), $N^\wedge(t^2)/t$
is non-increasing, and so $T$ is non-decreasing. Moreover, $T(x) / x^{1/2}$ is essentially
non-increasing (with a constant depending on $D$) on $[1,\infty)$, since this is equivalent to
showing that:
\[
 \frac{T(\log(1+1/t))}{\log^{1/2}(1+1/t)} = \frac{t^{1/2} N^\wedge(t)}{I_0(t)}
\]
is essentially non-decreasing (with the same constant) on $[0,1/(e-1)]$. The latter follows from
(\ref{eq:last-thing}) and Fact 3 on $J_0$.

Now all the conditions of Theorem \ref{thm:cap2} are fulfilled, except for the requirement that
$T(x)^2/x$ is properly non-increasing, but a quick look at Lemma 8 in \cite{BCRHard} reveals that
this may be relaxed to \emph{essentially} non-increasing, at a price of changing the explicit bounds
in the conclusion of Theorem \ref{thm:cap2}. We therefore deduce that
(\ref{eq:cap2-BCR}) is equivalent (up to a constant depending on $D$) to the Beckner-type inequality
(\ref{eq:Beckner}) which admits tensorization, and therefore also holds for the product structure
$(\Real^k,\abs{\cdot}^{\otimes k},\nu^{\otimes k})$. Applying Theorem \ref{thm:cap2}
again, we deduce that:
\[
 Cap_2(t,1/2) \geq c_D t^{1/2} T(\log(1+1/t)) = c_D N^\wedge(t) \;\;\; \forall t \in [0,1/2]
\]
holds for the product structure as well. Using
(\ref{eq:N-prop}) and our convexity assumptions (since the product measure $\nu^{\otimes k}$ is
log-concave on $\Real^k$), we may apply the direction $(1) \Rightarrow (2)$ of Theorem
\ref{thm:cap-iso}, and deduce that:
\[
 \tilde{I}_{(\Real^k,\abs{\cdot}^{\otimes k},\nu^{\otimes k})}(t) \geq c'_D t^{1/2} N^\wedge(t)
 \;\;\; \forall t \in [0,1/2] ~,
\]
which by (\ref{eq:last-thing}) implies:
\[
 \tilde{I}_{(\Real^k,\abs{\cdot}^{\otimes k},\nu^{\otimes k})}(t) \geq c''_D J(t) \;\;\; \forall t
\in [0,1/2] ~.
\]
This concludes the proof, up to the proof of Lemma \ref{lem:last-thing} below.
\end{proof}

\begin{lem} \label{lem:last-thing}
The estimate (\ref{eq:last-thing}) holds.
\end{lem}
\begin{proof}
We will show that:
\[
 1 \leq \brac{\int_t^\infty \frac{J_1(t)^2 ds}{J_1(s)^2 t}}^{1/2} \leq C D \;\;\; \forall t \in
[0,1] ~,
\]
for some constant $C>0$, which will conclude the proof.

The lower bound is immediate from the lower bound in
(\ref{eq:last-bounds}). For the upper bound, we decompose the integral into two parts:
\[
 \int_t^{1/2} \frac{J_1(t)^2 ds}{J_1(s)^2 t} + \int_{1/2}^\infty \frac{J_1(t)^2 ds}{J_1(s)^2 t} ~,
\]
with the first one interpreted as $0$ if $t> 1/2$. By Fact 1 and the definition of $J_1$, the second
integral can be estimated for $t \in [0,1]$ by:
\[
 \int_{1/2}^\infty \frac{J_1(t)^2 ds}{J_1(s)^2 t} \leq D^2 \int_{1/2}^\infty \frac{J_1(1)^2 ds}{4
J_1(1/2)^2 s^2} = 2 D^2 ~.
\]
To estimate the first integral for $t\in[0,1/2]$, we use the upper bound in (\ref{eq:last-bounds})
and the change of variables $v = \log(1+1/s)$:
\begin{eqnarray*}
\int_t^{1/2} \frac{J_1(t)^2 ds}{J_1(s)^2 t} &\leq& D^2 \int_t^{1/2} \frac{t \log(1+1/t) ds}{s^2
\log(1+1/s)}
= D^2 t \log(1+1/t) \int_{\log 3}^{\log(1+1/t)} \frac{\exp(v)}{v} dv \\
&\leq& C D^2 t \log(1+1/t) \frac{1+1/t}{\log(1+1/t)} \leq C' D^2 ~.
\end{eqnarray*}
This concludes the proof.
\end{proof}

\begin{rem}
In fact, by inspecting the bound given by Theorem 9 in
\cite{BCRHard} more carefully, one can repeat our argument for an
arbitrary isoperimetric profile $J$ (perhaps violating the
Central-Limit obstruction), and study what happens to the profile
under tensorization. We leave this for another note.
\end{rem}


\section{Approximation Argument} \label{sec:approx}

Recall that $\mu_m$ is said to converge to $\mu$ in total-variation if:
\[
\lim_{m \rightarrow \infty} \sup_{A \subset \Omega} \abs{\mu_m(A) - \mu(A)} = 0 ~.
\]
$\mu_m$ is said to converge to $\mu$ weakly if $\lim_{m \rightarrow \infty} \int f d\mu_m = \int f d\mu$ for any bounded continuous Borel function $f : \Omega \rightarrow \Real$.

In this section, we provide a careful approximation argument for deducing that our results from Section \ref{sec:semi-group} hold under arbitrary convexity assumptions, without requiring any further \emph{smoothness} conditions (as defined in the Introduction or more generally in Section \ref{sec:semi-group} and Remark \ref{rem:general-smoothness}). We recall that at this point, the proof of Theorem \ref{thm:Orlicz-12} is only valid under the additional smoothness conditions. We emphasize that this is not just a technical matter, and that our convexity assumptions will need to be invoked once again. To explain this better, let us describe a naive approximation approach which completely fails. Suppose that $(\Omega,d,\mu)$ satisfies a $(N,q)$ Orlicz-Sobolev inequality as in the assumption of Theorem \ref{thm:Orlicz-12}, and we would like to deduce from this the conclusion of this theorem, assuming that $(\Omega,d,\mu)$ satisfies our convexity assumptions. By definition, we know that there exists a sequence $\set{\mu_m}$ which approximates $\mu$ in total-variation, such that $(\Omega,d,\mu_m)$ satisfy our smooth convexity assumptions, and so the proof of Theorem \ref{thm:Orlicz-12} applies to these spaces. One may hope that since $\mu_m$ approximate $\mu$, the spaces $(\Omega,d,\mu_m)$ will also satisfy the $(N,q)$ Orlicz-Sobolev inequality (perhaps with a worse constant), allowing us to apply Theorem \ref{thm:Orlicz-12}.
Unfortunately, this is completely false in general. For instance, consider the measures $\mu_m$ which are uniform on the set $[0,1] \setminus [1/2-1/m,1/2+1/m]$, and converge to $\mu$, the uniform measure on $[0,1]$. Clearly, the spaces $(\Omega,d,\mu_m)$ ($m \geq 3$) do not satisfy any $(N,q)$ Orlicz-Sobolev inequality, whereas in the limit the space $(\Omega,d,\mu)$ will satisfy any reasonable inequality (Poincar\'e, log-Sobolev, etc.). We conclude that a different approach is needed.

\medskip

Our strategy in this section will be to show that the semi-group estimates of Section \ref{sec:semi-group} can be transferred to a setting without any smoothness assumptions. Our original argument, which at first relied on a method of weak-convergence due to Williams and Zheng \cite{WilliamsZheng} (see also Burdzy and Chen \cite{BurdzyChenWeakConvergence}), has been replaced by an elementary argument which we provide below. We continue with the notations used in Section \ref{sec:semi-group}, and recall the following definition:

\begin{dfn*}
A domain $\Omega \subset (M,g)$ is said to be locally convex, if all geodesics in $M$ tangent
to $\partial \Omega$ are locally outside of $\Omega$. By a result of Bishop
\cite{BishopInBayleRosales}, in case that $\Omega$ has $C^2$ boundary, this is equivalent to
requiring that the second fundamental form of $\partial \Omega$ with respect to the normal pointing
into $\Omega$ be positive semi-definite on all of $\partial \Omega$.
\end{dfn*}

Assume that $\mu_m$ converges in total-variation (and in particular, weakly) to an absolutely continuous probability measure $\mu$, so that $(\Omega,d,\mu_m)$ satisfy our smooth convexity assumptions. In other words, $d\mu_m = \exp(-\psi_m) dvol_M|_{\Omega_m}$ where $\Omega_m \subset M$ is a locally convex domain with $C^2$ boundary, $\psi_m \in C^2(\overline{\Omega}_m,\Real)$ and $Ric_g + Hess_g \psi_m \geq 0$ on $\Omega_m$.  Assume that $\mu$ is supported on a domain $\Omega$, and write $d\mu = \exp(-\psi) dvol_M|_{\Omega}$ for some Borel function $\psi : \Omega \rightarrow \Real$ (defined almost-everywhere on $\Omega$). By passing to an appropriate subsequence, we can assume that $\psi_m$ converges almost-everywhere on $M$ to $\psi$ (we implicitly extend the definition of these functions to $+\infty$ outside their original domains of definition).

\begin{lem} \label{lem:psi-ae}
One may always choose a version of $\psi$ which is locally Lipschitz on $\Omega$.
\end{lem}
\begin{proof}
Let $A \subset M$ denote the Borel subset of points $x \in M$ for which the sequence $\psi_m(x)$ converges to $\psi(x)$ (in the wide sense). We know that $vol_M(\Omega \setminus A) = 0$. We will show that for each $x_0 \in \Omega$, there exists a neighborhood $N_{x_0} \subset \Omega$ and a constant $C_{x_0}>0$, so that:
\begin{equation} \label{eq:lip-to-show}
\forall x,y \in N_{x_0} \cap A \;\;\; \abs{\psi(x) - \psi(y)} \leq C_{x_0} d(x,y) ~.
\end{equation}
Consequently, it will follow that one may extend $\psi$ by continuity from $A \cap \Omega$ to the entire $\Omega$, defining $\psi(z_0)$ for $z_0 \in \Omega \setminus A$ as $\psi(z_0) = \lim_{z \rightarrow z_0 , z \in A} \psi(z)$. The estimate (\ref{eq:lip-to-show}) will imply that this limit is well defined and that the resulting $\psi$ satisfies the same locally Lipschitz condition.

To deduce (\ref{eq:lip-to-show}), we will show that for any $x_0 \in \Omega$, there exists a geodesically convex neighborhood $N_{x_0} \subset \Omega$ of $x_0$, such that on $N_{x_0} \cap A$, $\psi$ coincides with a semi-convex function $\psi_{x_0} : N_{x_0} \rightarrow \Real$. By this we mean that there exists a smooth function $h_{x_0}$ on $N_{x_0}$ so that $\psi_{x_0} + h_{x_0} : N_{x_0} \rightarrow \Real$ is geodesically convex in $N_{x_0}$, meaning that it is convex on each geodesic in $N_{x_0}$. It is known (e.g. \cite[p. 642]{GreeneWuConvexApproximations}, \cite[(2.2)]{BangertGeneralizeAleksandrov}, \cite[Theorem 1.5.1]{Schneider-Book}) that geodesically convex functions are locally Lipschitz in the interior of the domain where they are finite, from which the same conclusion will hold for $\psi_{x_0}$, and consequently (\ref{eq:lip-to-show}) will hold for $\psi$.
In what follows, we refer to Greene and Wu \cite{GreeneWuConvexApproximations}, Bangert \cite{BangertGeneralizeAleksandrov}, Cordero-Erausquin, McCann and Schmuckenschl{\"a}ger \cite{CMSInventiones}, and the references therein, for further information on geodesically convex and semi-convex functions.

It is known that for any $x_0 \in M$, the geodesic open ball $B(x_0,r)$ for small enough $r>0$ is convex embedded in $M$, in the sense that it is both geodesically convex and that the exponential map $\exp_{x_0} : B_{T_{x_0} M}(0,r) \rightarrow B(x_0,r)$ is a diffeomorphism between $B_{T_{x_0} M}(0,r) \subset T_{x_0} M$ and $B(x_0,r) \subset M$. We will therefore choose our neighborhood $N_{x_0}$ to be a convex embedded ball $B(x_0,r)$, so that in addition $\overline{B(x_0,r)}$ is contained in $\Omega$. Since $\psi_m$ converge to $\psi$ almost everywhere, it is clear that if $\overline{B(x_0,r)} \subset \Omega$ then $B(x_0,r)$ must be contained in $\Omega_m$ for all $m \geq m_{x_0}$ (we could add ``apart from a subset of zero measure'' for safety, but this is in fact not necessary due to the convexity of the domains).

Choosing $r>0$ small enough, it is known (e.g. \cite[p. 643]{GreeneWuConvexApproximations}, \cite[p. 311]{BangertGeneralizeAleksandrov}) that $f_{x_0} := d(x_0,\cdot)^2$ is a $C^\infty$ function on $B(x_0,r)$ whose Riemannian Hessian satisfies $Hess_g f_{x_0} \geq A_{x_0} g$ on $B(x_0,r)$ for some $A_{x_0} > 0$.
Denoting:
\[
R_{x_0,r} := \max \set{Ric_g(v,v) ; v \in T_x M , g(v,v) = 1 , x \in \overline{B(x_0,r)} },
\]
we define $h_{x_0} = \frac{\max(R_{x_0,r},0)}{A_{x_0}} f_{x_0}$.

Since for all $m \geq m_{x_0}$, $Ric_g + Hess_g \psi_m \geq 0$ on $B(x_0,r)$, it follows from the above construction that $Hess_g (\psi_m + h_{x_0}) \geq 0$ on $B(x_0,r)$. Since $\psi_m + h_{x_0} \in C^2(B(x_0,r))$, it is known (\cite[p. 310]{BangertGeneralizeAleksandrov}) that this is equivalent to being geodesically convex in $B(x_0,r)$.
We now employ \cite[Theorem 10.8]{RockafellarBook}, whose proof easily passes to the Riemannian setting (taking into account a slight modification provided in \cite[Lemma 2.1]{BangertGeneralizeAleksandrov} of a Euclidean argument). This theorem asserts that if a sequence of geodesically convex functions $\set{f_i}$ pointwise converges on a dense subset of a geodesically convex open set $N$ (to a finite value in each point of the subset), then the pointwise limit in fact exists for each $x \in N$, and the function $f(x) := \lim_{i \rightarrow \infty} f_i(x)$ is finite and geodesically convex on $N$. Since $\psi_m + h_{x_0}$ converges to the finite function $\psi + h_{x_0}$ on $B(x_0,r) \cap A$, it follows that in fact $\psi_m + h_{x_0}$ converges to a geodesically convex function on the entire $B(x_0,r)$. Writing this function as $\psi_0 + h_{x_0}$, we realize that $\psi_0$ coincides with $\psi$ on $B(x_0,r) \cap A$. The argument is therefore complete.
\end{proof}

\begin{rem}
In fact, we have shown that $\psi_m$ converge pointwise on all of $\Omega$, and that the limit $\psi = \lim_{m \rightarrow \infty} \psi_m$ is a semi-convex function (in the sense that in a small enough geodesically convex neighborhood, we may add to it a smooth function to obtain a geodesically convex function in that neighborhood).
\end{rem}

\begin{rem} \label{rem:psi-ae}
We will henceforth choose $\psi$ as constructed in Lemma \ref{lem:psi-ae}. This implies by Rademacher's theorem that:
\begin{enumerate}
\item The differential $\nabla \psi$ exists almost everywhere on $\Omega$.
\item $\abs{\nabla \psi}$ is bounded almost everywhere on compact subsets of $\Omega$.
\end{enumerate}
\end{rem}

Let $X^{(m)} : \Lambda_m \times [0,T] \rightarrow M$ denote the diffusion process on $\overline{\Omega}_m$ with reflection on the boundary generated by $\Delta_{(\Omega_m,\mu_m)}$, defined on the probability space $(\Lambda_m,\F_m,\P_m)$ and some fixed $T>0$. Let $P^{(m)}_t$ for $t \in [0,T]$ denote the semi-group associated to $X^{(m)}$. We will assume that the initial distribution of $X^{(m)}(0)$ is given by the stationary measure $\mu_m$, so that $X^{(m)}$ is a stationary process.

Using a forward-backward martingale decomposition due to Lyons and Zheng \cite{LyonsZhengForwardBackwardDecomposition} and a tightness criterion for stochastic processes with continuous paths, it can be shown as in \cite[p. 472]{WilliamsZheng} (see also \cite[p. 31]{BurdzyChenWeakConvergence}, \cite[pp. 248-257]{FukushimaBook}) with a minor adaptation to the Riemannian setting,
that there exists a subsequence $X^{(m_k)}$ which converges weakly (as measures on $[0,T] \times M$ endowed with the locally uniform topology) to some process $X$. By passing to a subsequence, let us assume that $X^{(m)}$ converges weakly to $X$. In particular, for any fixed $t \in [0,T]$, the law of $X^{(m)}(t)$ weakly converges to that of $X(t)$, and hence $X$ is also stationary with stationary measure $\mu$. Since $X(0)$ is by definition distributed according to $\mu$, there is a one-to-one correspondence between the spaces $L_2(\mu) := L_2(\Omega,\B(\Omega),\mu)$ and $L_2(\Lambda,\sigma(X(0)),\P)$, where $\sigma(X(0))$ is the $\sigma$-field generated by $X(0)$ and $X$ is defined on the space $\Lambda$ with probability measure $\P$. Consequently, we can define for any $t \in [0,T]$ the following (bounded) linear operator $P_t$ on $L_2(\mu)$:
\begin{equation} \label{eq:Pt-def}
 P_t f(X(0)) = E(f(X(t)) | X(0)) \;\;\; \forall f \in L_2(\mu) ~.
\end{equation}
In fact, as in \cite{WilliamsZheng}, it should be possible to show that $X$ is a continuous Markov process, that $P_t$ is a strongly continuous semi-group associated to it, and that the associated Dirichlet form is exactly given by:
\[
\E(f,g) := \int_\Omega \scalar{\nabla f,\nabla g} d\mu ~.
\]

However, we will not require all this information. We will only use the weak convergence (through a subsequence) of the processes $\set{X^{(m)}(0),X^{(m)}(t)}$, defined on the 2-point set $\set{0,t}$, to $\set{X(0),X(t)}$. We provide an elementary argument to deduce the tightness of this sequence (from which the former statement follows by Prokhorov's Theorem). Fixing a point $x_0 \in M$, since the process $X^{(m)}$ is stationary:
\begin{eqnarray*}
& & \P_m(d(X^{(m)}(0),x_0) > R \text{ or } d(X^{(m)}(t),x_0)>R) \\
&\leq& \P_m(d(X^{(m)}(0),x_0) > R) + \P_m(d(X^{(m)}(t),x_0) > R) = 2 \mu_m(d(x,x_0) > R) ~,
\end{eqnarray*}
so it remains to show that:
\[
 \lim_{R \rightarrow \infty} \sup_m \mu_m \set{d(x,x_0) > R} = 0 ~,
\]
which is easily seen to hold, since this holds for the (probability)
measure $\mu$, and $\mu_m$ converge to $\mu$ in total-variation.

We conclude that by passing to a subsequence, we may assume that $\set{X^{(m)}(0),X^{(m)}(t)}$ converges weakly to $\set{X(0),X(t)}$. In other words, for any $f,g$ continuous and bounded on $M$:
\[
 \lim_{m \rightarrow \infty} E(f(X^{(m)}(t)) g(X^{(m)}(0))) = E(f(X(t)) g(X(0))) ~,
\]
or equivalently:
\begin{equation} \label{eq:Pt-approx}
 \lim_{m \rightarrow \infty} \int (P^{(m)}_t f) g d\mu_m = \int (P_t f) g d\mu ~,
\end{equation}
where $P_t$ is the linear operator defined in (\ref{eq:Pt-def}). Clearly, this also extends to hold for all $f \in L_\infty(\mu)$ and $g \in L_1(\mu)$.

Using (\ref{eq:Pt-approx}), we can now transfer the known estimates (\ref{eq:L_q-bound}) and Corollary \ref{cor:Ledoux} on the semi-groups $P^{(m)}_t$ to the operators $P_t$. Indeed, given a bounded and smooth function $f$ on $\Omega$, we need to show that:
\begin{equation} \label{eq:Pt-1}
 \norm{f - P_t f}_{L_1(\mu)} \leq \sqrt{2t} \norm{ \abs{\nabla f} }_{L_1(\mu)} ~,
\end{equation}
and that:
\begin{equation} \label{eq:Pt-2}
\forall q \in [2,\infty] \;\;\; \norm{\abs{\nabla P_t f}}_{L_q(\mu)} \leq \frac{1}{\sqrt{2t}} \norm{f}_{L_q(\mu)} ~,
\end{equation}
using the same estimates (\ref{eq:Pt-1}) and (\ref{eq:Pt-2}) for $P^{(m)}_t$ and $\mu_m$ replacing $P_t$ and $\mu$, respectively (the latter are known to be true as described in Section \ref{sec:semi-group}, after approximating $f$  with functions in $\B(\Omega_m)$). Note that we interpret the second estimate (\ref{eq:Pt-2}) regarding $\nabla P_t f$ in the sense of distributions as described below, since this is all that is needed for the applications of Section \ref{sec:semi-group}. Writing:
\[
 \norm{f - P_t f}_{L_1(\mu)} = \sup\set{ \int (f - P_t f) g d\mu \; ; \; \norm{g}_{L_\infty} \leq 1 } ~,
\]
the estimate (\ref{eq:Pt-1}) immediately follows from the same estimate for $P^{(m)}_t$ and $\mu_m$ and the weak convergence (\ref{eq:Pt-approx}). The estimate (\ref{eq:Pt-2}) is harder to handle, since it involves the distributional gradient of $P_t(f)$. Setting $p = q^*$, we interpret:
\[
 \norm{\abs{\nabla P_t f}}_{L_q(\mu)} := \sup\set{\int \scalar{\nabla P_t f , g} d\mu \; ; \; \begin{array}{c} \text{$g$ is a smooth vector field with} \\ \text{compact support in $\Omega$} \; , \; \int |g|^p d\mu \leq 1 \end{array} } ~,
\]
and $\int \scalar{\nabla P_t f , g} d\mu$ is interpreted in the distributional sense (using integration by parts).
Let $g$ denote such a vector field as above. Since $\mu_m$ converges to $\mu$ it total-variation, it remains to show the first equality in:
\begin{eqnarray*}
\int \scalar{\nabla P_t f , g} d\mu &=& \lim_{m \rightarrow \infty} \int \sscalar{\nabla P^{(m)}_t f , g} d\mu_m \leq
\lim_{m \rightarrow \infty}  \snorm{|\nabla P^{(m)}_t f|}_{L_q(\mu_m)} \norm{\abs{g}}_{L_p(\mu_m)} \\
&\leq& \lim_{m \rightarrow \infty} \frac{1}{\sqrt{2t}} \norm{f}_{L_q(\mu_m)} =  \frac{1}{\sqrt{2t}} \norm{f}_{L_q(\mu)} ~.
\end{eqnarray*}
Let us write:
\begin{eqnarray*}
& & \abs{ \int \sscalar{\nabla P_t f , g} d\mu - \int \sscalar{\nabla P^{(m)}_t f , g} d\mu_m } \\
&\leq&
\abs{ \int \sscalar{\nabla P_t f , g} d\mu - \int \sscalar{\nabla P^{(m)}_t f , g} d\mu } +
\abs{ \int \sscalar{\nabla P^{(m)}_t f , g} d\mu - \int \sscalar{\nabla P^{(m)}_t f , g} d\mu_m } ~.
\end{eqnarray*}
Since $\snorm{|\nabla P^{(m)}_t f|}_{L_\infty} \leq \frac{1}{\sqrt{2t}} \norm{f}_{L_\infty}$ for all $m$, the convergence of $\mu_m$ to $\mu$ in total-variation implies that the second term converges to 0 as $m \rightarrow \infty$. To handle the first term, we intergrate by parts:
\begin{eqnarray*}
& & \abs{ \int \sscalar{\nabla P_t f , g} d\mu - \int \sscalar{\nabla P^{(m)}_t f , g} d\mu } =
\abs{ \int (P_t f - P_t^{(m)}f ) (\nabla \cdot g - \scalar{ g , \nabla \psi } ) d\mu } \\
&\leq&
\abs{ \int P_t f (\nabla \cdot g - \scalar{ g , \nabla \psi } ) d\mu -
      \int P^{(m)}_t f (\nabla \cdot g - \scalar{ g , \nabla \psi } ) d\mu_m } \\
&+&
\abs{ \int P^{(m)}_t f (\nabla \cdot g - \scalar{ g , \nabla \psi } ) d\mu_m -
      \int P^{(m)}_t f (\nabla \cdot g - \scalar{ g , \nabla \psi } ) d\mu } ~.
\end{eqnarray*}
It follows from Lemma \ref{lem:psi-ae} and Remark \ref{rem:psi-ae} that $\scalar{g, \nabla \psi}$ is a bounded function. This implies that the first term converges to 0 by $(\ref{eq:Pt-approx})$, and since $\snorm{P^{(m)}_t f}_{L_\infty} \leq \norm{f}_{L_\infty}$, the convergence of $\mu_m$ to $\mu$ in total-variation implies that the second term converges to 0 as well.

We conclude that the estimates (\ref{eq:Pt-1}) and (\ref{eq:Pt-2}) hold for the linear operators $P_t$ defined using the limiting process $X$, and we may use these operators in place of the diffusion semi-group in the relevant parts of the proof of Theorem \ref{thm:Orlicz-12} (and consequently Theorems \ref{thm:Orlicz-12-strong}, \ref{thm:Orlicz-summary} and \ref{thm:cap-iso}), so that the conclusion of these theorems remains valid for $\mu$ as above.

\setlinespacing{0.8} 
\setlength{\bibspacing}{0pt}

\bibliographystyle{plain}
\def\cprime{$'$}

\end{document}